\documentclass[12pt]{amsart}
\usepackage[psamsfonts]{amssymb} 
\usepackage{amsfonts,amsmath}
\usepackage[all]{xy}
\usepackage{stmaryrd}
\usepackage{enumerate}

\usepackage{color}
\makeatletter
\newtheorem*{rep@theorem}{\rep@title}
\newcommand{\newreptheorem}[2]{%
\newenvironment{rep#1}[1]{%
 \def\rep@title{#2 \ref{##1}}%
 \begin{rep@theorem}}%
 {\end{rep@theorem}}}
\makeatother

\newtheorem{theorem}{Theorem}[section]
\newreptheorem{theorem}{Theorem}
\newtheorem{prop}[theorem]{Proposition}
\newtheorem{proposition}[theorem]{Proposition}
\newtheorem{lemma}[theorem]{Lemma}
\newtheorem{corollary}[theorem] {Corollary}
\newreptheorem{corollary}{Corollary}

\theoremstyle{remark}
\newtheorem{remark}[theorem]{Remark}

\theoremstyle{definition}
\newtheorem{definition}[theorem]{Definition}

\def\C{\mathcal{C}}
\def\G{\Gamma}
\def\g{\gamma}
\def\ad{{\rm{ad}}}
\def\aut{{\rm{Aut}}}

\def\autn{{\rm{Aut}}(F_n)}
\def\autgd{\aut(\mathcal{G},D_{*v})}
\def\out{{\rm{Out}}}
\def\outn{{\rm{Out}}(F_n)}
\def\outgd{\aut(\mathcal{G},\widehat{D})}
\def\autog{\aut^0(\mathcal{G},\widehat{D})}

\def\co{\colon\thinspace}
\DeclareMathOperator{\im}{Im}
\def\id{{\rm{id}}}

\title{Centralisers of Dehn twist automorphisms of free groups}
\author{Moritz Rodenhausen and Richard D. Wade}
\begin{document}

\begin{abstract} 
We refine Cohen and Lustig's description of centralisers of Dehn twists of free groups. We show that the centraliser of a Dehn twist of a free group has a subgroup of finite index that has a finite classifying space. We describe an algorithm to find a presentation of the centraliser. We use this algorithm to give an explicit presentation for the centraliser of a Nielsen automorphism in $\aut(F_n)$. This gives restrictions to actions of $\aut(F_n)$ on CAT(0) spaces.
\end{abstract}

\maketitle

\section{Introduction}

Given a group $G$ and an element $g \in G$, a natural question is study the centraliser $C(g)$ of $g$ in $G$. In several classes of groups, such as hyperbolic and CAT(0) groups \cite{BH}, as well as mapping class groups \cite{Iv}, centralisers of elements are reasonably well-understood. In $\out(F_n)$, Feighn and Handel classified abelian subgroups in $\out(F_n)$ by studying centralisers of elements \cite{FH} and the centraliser of a \emph{fully irreducible} element is virtually-cyclic \cite{BFH}. However, it is not clear how centralisers of elements behave in general.

A \emph{Dehn twist} $D=(\mathcal{G},(\g_e)_{e \in E(\G)})$ is determined by a graph of groups $\mathcal{G}$ along with a twisting element $\g_e$ in the centre of each edge group of $\mathcal{G}$. An isomorphism $\rho:\pi_1(\mathcal{G},v) \to F_n$ then defines elements $D_{*v} \in \aut(F_n)$ and $\widehat{D} \in \out(F_n)$. The class of all such automorphisms includes Nielsen automorphisms and Whitehead automorphisms of infinite order. More generally, every automorphism of linear growth has a power which is a Dehn twist in $\out(F_n)$ \cite{KLV}. 

It is illuminating to look first at the analogous elements in mapping class groups: comparable to $\widehat{D}$ is a \emph{multitwist} $M$ given by taking disjoint, pairwise non-isotopic, simple closed curves $c_1,\ldots,c_m$ on a surface $\Sigma$, non-zero integers $a_1,\ldots,a_m$, and defining $M=T_{c_1}^{a_1} \cdots T_{c_m}^{a_m}$, where $T_{c_i}$ is the Dehn twist (in the classical sense) in the curve $c_i$. Cutting along the curves $c_1,\ldots,c_m$ gives a set of punctured subsurfaces $\Sigma_1,\ldots,\Sigma_k$. The centraliser $C(M)$ of the multitwist in the mapping class group $MCG(\Sigma)$ has a finite index subgroup $C^0(M)$ which fixes each curve and subsurface up to isotopy. $C^0(M)$ fits into the exact sequence:
$$1 \to \mathbb{Z}^m \to C^0(M) \to \bigoplus_{i=1}^k MCG(\Sigma_k) \to 1,$$
where the free abelian group $\mathbb Z^m$ in this sequence is generated by $T_{c_1},\ldots, T_{c_m}$ and the right hand surjection is given by restricting the mapping class to each $\Sigma_i$. This allows one to study $C(M)$ through mapping class groups of open subsurfaces.

Our main theorem establishes a similar picture when we have an \emph{efficient} Dehn twist $D=(\mathcal{G},(\g_e)_{e \in E(\G)})$ (Definition~\ref{efficient}, below) and $\pi_1(\mathcal{G},v)\cong F_n$. In this situation, each vertex group $G_w$ of $\mathcal{G}$ is a finitely generated free group and each edge group $G_e$ is cyclic. For a vertex group $G_w$, the generating elements of the edge groups adjacent to $w$ give a set of conjugacy classes $\mathcal{C}_w$ in $G_w$. Rather than having mapping class groups of punctured subsurfaces, we instead have the relative automorphism groups  $\out(G_w,\mathcal{C}_w)$ and $\aut(G_w,\mathcal{C}_w)$ that consist of (outer) automorphisms that fix each conjugacy class in the finite set $\mathcal{C}_w$.

There is an action of $C(\widehat{D})$ on the underlying graph $\G$ of $\mathcal{G}$, and we define $C^0(\widehat{D})$ to be the finite-index subgroup consisting of automorphisms that act trivially on $\G$. 

\begin{reptheorem}{t:ses}
Let $D$ be an efficient Dehn twist on a graph of groups $\mathcal{G}$ with $\pi_1(\mathcal{G},v)\cong F_n$. Let $C(\widehat{D})$ be the centraliser of $D$ in $\outn$. There exists a homomorphism: \begin{align*} \bar{\alpha}\co &C(\widehat{D}) \to \aut(\G),\end{align*} with kernel $C^0(\widehat{D})$ a finite index subgroup fitting into the exact sequence: \begin{gather*}1\to DO(\mathcal{G}) \to C^0(\widehat{D}) \to \bigoplus_{w \in V(\G)} \out(G_w,\mathcal{C}_w)\to1,  \end{gather*}
where $DO(\mathcal{G})$ is a free abelian group of Dehn twists of rank equal to the number of geometric edges of $\mathcal{G}$.\end{reptheorem}

Our main inputs are the theory of \emph{automorphisms of graphs of groups} developed by Bass and Jiang \cite{BJ} (summarised in Section~\ref{s:background}) and a theorem of Cohen and Lustig \cite{CL} showing that an element of the centraliser of an efficient Dehn twist may be represented by an automorphism of $\mathcal{G}$.
There is a group $\aut(\mathcal{G},\widehat{D})$ of automorphisms of the graph of groups $\mathcal{G}$ and a surjection $\pi:\aut(\mathcal{G},\widehat{D}) \to C(\widehat{D})$. The maps from Theorem~\ref{t:ses} are easy to define on $\aut(\mathcal{G},\widehat{D})$, and the work of Bass and Jiang shows that these homomorphisms factor through $\pi$. The same techniques have been used by Levitt \cite{L} to study automorphisms of hyperbolic groups, and more recently Guirardel--Levitt \cite{GL} in the relatively hyperbolic case.

Cohen and Lustig also show that each Dehn twist in $\out(F_n)$ may be represented by an efficient Dehn twist. As each group $\out(G_w,\mathcal{C}_w)$ has a finite index subgroup with finite classifying space (\cite{CV}, Corollary 6.1.4.), we have the following corollary:
\begin{repcorollary}{c:fpout} If $\phi \in \out(F_n)$ is a Dehn twist automorphism then $C(\phi)$ has a finite-index, torsion-free subgroup with finite classifying space. \end{repcorollary}

In $\aut(F_n)$ the situation is slightly trickier, as here it is not true that every Dehn twist has an efficient representative.  For this reason, in Section~\ref{s:peff} we introduce the notion of a \emph{pointedly efficient Dehn twist}. We show that every Dehn twist in $\aut(F_n)$ has a pointedly efficient representative, and for such an element we have a similar decomposition as in Theorem~\ref{t:ses}. We use this in Section~\ref{s:alg} to describe an algorithm to find a presentation of the centraliser of a Dehn twist automorphism in $\aut(F_n)$ or $\out(F_n)$. In Section~\ref{s:Nielsen} we use this algorithm to give a presentation for the centraliser of a Nielsen automorphism $\rho$ in $\aut(F_n)$. This allows us to compute $H_1(C(\rho))$, which after some interesting low rank cases, stabilizes for $n\geq5$:

\begin{repcorollary}{pres:cab} Let $\rho\in\aut(F_n)$ be a Nielsen automorphism. Then
$$H_1(C(\rho))\cong \begin{cases}
\mathbb{Z}^2\oplus\mathbb{Z}/2\mathbb{Z} ,&\text{if } n=2,\\
\mathbb{Z}\oplus(\mathbb{Z}/2\mathbb{Z})^3 ,&\text{if } n=3,\\
(\mathbb{Z}/2\mathbb{Z})^3 ,&\text{if } n=4,\\
(\mathbb{Z}/2\mathbb{Z})^2 ,&\text{if } n\ge5.
\end{cases}$$
When $n=2$, the class $\llbracket\rho\rrbracket$ is a primitive element of $\mathbb{Z}^2$, when $n=3$ it is twice a generator of $\mathbb{Z}$, and otherwise $\llbracket\rho\rrbracket=0$.
\end{repcorollary}

Corollary~\ref{pres:cab} has an application to actions of $\aut(F_n)$ on CAT(0) spaces. With such actions, elements with non-zero \emph{translation length} have infinite order in the abelianisation of their centraliser.

\begin{repcorollary}{c:transl}
If $n\ge4$, Nielsen automorphisms always act by zero translation length whenever $\aut(F_n)$ acts isometrically on a proper CAT(0) space.
\end{repcorollary}

This improves on a result of Bridson, who showed the above for $n \geq 6$. Furthermore, Bridson \cite{B:rd} describes actions of $\aut(F_3)$ on CAT(0) spaces where Nielsen automorphisms have positive translation length. Hence the requirement that $n \geq 4$ in this corollary is as strong as possible. 

The authors would like to thank Martin Bridson for his advice and encouragement on this paper. The first author also wants to thank his PhD advisor Carl-Friedrich B\"odigheimer and the International Max Planck Research School for Moduli Spaces (IMPRS) in Bonn, Germany, for the support of his visit at the University of Oxford in 2011. The second author was supported by the EPSRC of Great Britain and the University of Utah. Furthermore, both authors would like to thank the referee for a detailed and considered referee report.

\section{Background}\label{s:background}

This section consists of background material on graphs of groups and their automorphisms. We take most our notation from~\cite{CL}. These concepts are also defined in \cite{B:ct} with slightly different notation.
\subsection{Graphs of groups}

A \emph{graph of groups} $\mathcal{G}$ is a tuple $$\mathcal{G}=(\Gamma, (G_v)_{v \in V(\Gamma)}, (G_e)_{e \in E(\Gamma)}, (f_e)_{e \in E(\G)})$$ such that:

\begin{itemize}
\item $\G$ is a finite, connected graph in the sense of Serre (cf. I \textsection2.1 in \cite{S}) with vertex set $V(\Gamma)$ and edge set $E(\G)$. 
\item Each $G_e$, $G_v$ is a group.
\item If $\tau(e)$ is the terminal vertex of an edge $e$, we have an injective \emph{edge homomorphism} $f_e\co G_e \to G_{\tau(e)}$.
\item For any edge $e$, we have $G_e=G_{\bar{e}}$, where $\bar e$ denotes the edge $e$ with reversed orientation.\end{itemize}

We let $\iota(e)=\tau(\bar{e})$ denote the initial vertex of an edge $e$.

\subsection{The path group and related subsets}

The \emph{path group} of $\mathcal{G}$, denoted $\Pi(\mathcal{G})$, is defined by taking the free group $F$ generated by the letters $(t_e)_{e \in E(\G)}$ and quotienting out the free product $(\ast_{v\in V(\G)} G_v) \ast F$ by the relations:

\begin{itemize}
\item $t_e=t_{\bar{e}}^{-1}$ for all $e \in E(\G)$,
\item $t_ef_e(a)t_e^{-1}=f_{\bar{e}}(a)$ for all $e \in E(\G)$ and $a \in G_e$.
\end{itemize}

We say that an element $g \in \Pi(\mathcal{G})$ is connected if there exists a (possibly trivial) path $e_1,\ldots,e_k$ in $\G$ starting from a vertex $v_0$ and elements $g_0,g_1,\ldots,g_k$ such that $g_0 \in G_{v_0}$, $g_i \in G_{\tau(e_i)}$ for each $i\ge1$ and: $$g=g_0t_{e_1}g_1t_{e_2} \cdots g_{k-1}t_{e_k}g_k.$$ We define $\pi_1(\mathcal{G},v,w)$ to be the set of elements of $\Pi(\mathcal{G})$ represented by connected words whose underlying paths start at $v$ and end at $w$. If $v=w$, the set forms a subgroup of $\Pi(\mathcal G)$ -- the \emph{fundamental group of the graph of groups} -- and is denoted $\pi_1(\mathcal{G},v)$.

Given any element $x$ of a group $G$, let $\ad_x$ be the inner automorphism given by the map $g \mapsto xgx^{-1}$. In this paper, automorphisms always act on the left (so that $\ad_{xy}=\ad_x\ad_y$). If $W \in \pi_1(\mathcal{G},v,w)$ then the restriction of $\ad_W\co \Pi(\mathcal{G}) \to \Pi(\mathcal{G})$ to $\pi_1(\mathcal{G},w)$ induces an isomorphism between $\pi_1(\mathcal{G},w)$ and $\pi_1(\mathcal{G},v)$.

\subsection{Automorphisms of graphs of groups.}

Let $\mathcal{G}$ be a graph of groups. An \emph{automorphism of $\mathcal{G}$} is a tuple of the form $$(H_\G,(H_v)_{v \in V(\G)},(H_e)_{e \in E(\G)}, (\delta(e))_{e \in E(\G)}),$$ where

\begin{itemize}
\item $H_\G\co \G \to \G$ is a graph automorphism,
\item $H_v\co G_v \to G_{H_\G(v)}$ is a group isomorphism,
\item $H_e=H_{\bar{e}}\co G_e \to G_{H_\G(e)}$ is a group isomorphism,
\item $\delta(e)$ is an element of $G_{\tau(H_\G(e))}$,
\end{itemize}

with the additional compatibility requirement that
\begin{equation}\label{hed1}
H_{\tau(e)}(f_e(a))=\delta(e)f_{H_\G(e)}(H_e(a))\delta(e)^{-1}
\end{equation}
for all $e \in E(\G)$ and $a \in G_e$. We shall often look at the case where $H_\G$ is the identity on $\G$ and $H_e$ is the identity map on each edge group. Here the compatability requirement can be phrased in the simpler sense, that:
\begin{equation}\label{hed2} H_{\tau(e)}(f_e(a))=\delta(e)f_e(a)\delta(e)^{-1} \end{equation} for all edges $e \in E(\G)$ and $a \in G_e$. 

\subsection{The automorphism group of $\mathcal{G}$}

The set of automorphisms of $\mathcal{G}$ forms a group, which we call $\aut(\mathcal{G})$. If $H$ and $H'$ are two automorphisms of $\mathcal{G}$, then their product $H''=H\circ H'$ is defined as follows
(for simplicity we write $H(v)$ and $H(e)$ instead of $H_\G(v)$ and $H_\G(e)$): 
\begin{align*} H''_\G&=H_\G H_\G', \\ H_e''&=H_{H'(e)}H'_e,\displaybreak[0]\\ H_v''&=H_{H'(v)}H_v', \\ \delta''(e)&=H_{H'(\tau(e))}(\delta'(e))\delta(H'(e)).\end{align*}

This operation is associative and has an identity element where $H_\G,H_v,H_e$ are all the identity automorphisms and each $\delta(e)$ is trivial. Furthermore, each $H \in \aut(\mathcal{G})$ has an inverse $H^{-1}$ defined by taking $(H^{-1})_\G=(H_\G)^{-1}$, $(H^{-1})_e=(H_{H_\G^{-1}(e)})^{-1}$, $(H^{-1})_v=(H_{H_\G^{-1}(v)})^{-1}$, and $\delta^{-1}(e)=H^{-1}_{H^{-1}(\tau(e))}(\delta(H^{-1}(e))^{-1})$. 

\subsection{The action of $\aut(\mathcal{G})$ on $\Pi(\mathcal{G})$} \label{s:action}

An element $H\in \aut(\mathcal{G})$ induces an automorphism $H_*\co \Pi(\mathcal{G}) \to \Pi(\mathcal{G})$ by taking:
\begin{align*} g &\mapsto H_v(g), &&g \in G_v,\\
t_e &\mapsto \delta(\bar{e})t_{H(e)}\delta(e)^{-1}, &&t_e \in F.
\end{align*}
The map $H_*$ takes connected words to connected words, so for each vertex $v$ of $\Gamma$ there is an induced map $H_{*v}\co \pi_1(\mathcal{G},v) \to \pi_1(\mathcal{G},H_\G(v))$. If $H_\G(v)=v$, then $H_{*{v}} \in \aut(\pi_1(\mathcal{G},v))$. Similarly, if $H_\G(v)=w$, we can choose an element $W \in \pi_1(\mathcal{G},v,w)$ so that $\ad_W H_{*v} \in  \aut(\pi_1(\mathcal{G},v))$. If $W, W' \in \pi_1(\mathcal{G},v,w)$ then $\ad_WH_{*v}$ and $\ad_{W'}H_{*v}$ differ by $\ad_{WW'^{-1}}$ in $\aut(\pi_1(\mathcal{G},v))$, so $H$ determines an element $\widehat{H}$ of $\out(\pi_1(\mathcal{G},v))$. 

We denote by $\aut(\G)$ the group of automorphisms of the graph $\G$, and by $\aut(\G,v)$ the subgroup of graph automorphisms fixing the base vertex $v$. Let $\aut(\mathcal{G},v)$ be the subgroup of $\aut(\mathcal{G})$ consisting of elements such that $H_\G\in\aut(\G,v)$. The next lemma follows from the discussion above and the definition of multiplication in $\aut(\mathcal{G})$.

\begin{lemma}\label{l:aut} The map $H \mapsto \widehat{H}$ induces a homomorphism $$U\co \aut(\mathcal{G}) \to \out(\pi_1(\mathcal{G},v))$$ and the map $H \mapsto H_{*v}$ induces a homomorphism \[ \pushQED{\qed} V\co \aut(\mathcal{G},v) \to \aut(\pi_1(\mathcal{G},v)). \qedhere\popQED \]\end{lemma}

The image of the map $U$ in $\out(\pi_1(\mathcal{G},v))$ can be thought of as the automorphisms which fix the splitting of $\pi_1(\mathcal{G},v)$ given by $\mathcal{G}$.  Bass and Jiang \cite{BJ} study $\im U$ by making use of the action of $\pi_1(\mathcal{G},v)$ on the Bass--Serre tree $T_\mathcal{G}$ associated to $\mathcal{G}$. They focus on the situation where the action is \emph{minimal} (so has no invariant subtrees) and non-abelian (the action does not fix any end of $T_\mathcal{G}$). They show the following:

\begin{theorem}[\cite{BJ}, Theorem 6.4] \label{t:BJ}
Suppose that the action of $\pi_1(\mathcal{G},v)$ on the Bass--Serre tree $T_\mathcal{G}$ is minimal and non-abelian. Let $H \in \ker U$. Then:

\begin{enumerate}
\item $H_\G = \id_{\G}$.
\item For each vertex $w \in V(\G)$ and each edge $e \in E(\G)$ we have $H_w= \ad(g_w)$ and $H_e=\ad(g_e)$ for some $g_w \in G_w$ and $g_e \in G_e$.
\end{enumerate}
\end{theorem}
These conditions are not sufficient to determine when $H \in \ker U$, however item (1) implies the following:

\begin{proposition} \label{p:graph}
Suppose that the action of $\pi_1(\mathcal{G},v)$ on the Bass--Serre tree $T_\mathcal{G}$ is minimal and non-abelian. The homomorphism $$\alpha \co \aut(\mathcal{G}) \to \aut(\G)$$ given by $H \mapsto H_\G$ descends to a homomorphism \[ \pushQED{\qed} \bar\alpha \co \im U \to \aut(\G).\qedhere\popQED\]
\end{proposition}

Let $\aut^0(\mathcal{G})=\{H \in \aut(\mathcal{G}):H_\G=1\}$ be the kernel of the map $\alpha$, and $\im U ^0$ the kernel of the map $\bar \alpha$ (equivalently, $\im U ^0$ is the image of $\aut^0(\mathcal{G})$ in $\out(\pi_1(\mathcal{G},v))$. Item (2) in Theorem \ref{t:BJ} implies:

\begin{proposition}\label{p:BJ2}
Suppose that the action of $\pi_1(\mathcal{G},v)$ on the Bass--Serre tree $T_\mathcal{G}$ is minimal and non-abelian. The homomorphism $$ A \co \aut^0(\mathcal{G}) \to \bigoplus_{w\in V(\G)}\out(G_w) $$ given by $H \mapsto (H_w)_{w \in V(\G)}$ descends to a homomorphism \[ \pushQED{\qed}  \bar A \co \im U ^0 \to \bigoplus_{w\in V(\G)}\out(G_w). \qedhere\popQED\] \end{proposition}

The kernel of the map $\bar A$ can still be quite complicated, and is described by Bass and Jiang by a 4 term filtration. However, in our situation the kernel will be much simpler, and is describable only in terms of \emph{Dehn twist automorphisms}, which we discuss in the next section.

\subsection{Dehn twist automorphisms.} \label{s:dtdef}

\begin{definition}
An automorphism $D$ of a graph of groups $\mathcal{G}$ is called \emph{Dehn twist} if:
\begin{itemize}
\item $D_\G=\id_{\G}$,
\item $D_w=\id_{G_w}$ for all $w\in V(\G)$,
\item $D_e=\id_{G_e}$ for all $e\in E(\G)$,
\item there is an element $\gamma_e$ in the centre $Z(G_e)$ of each edge group such that $\delta(e)=f_e(\gamma_e)$.
\end{itemize}
\end{definition}

Every collection $(\gamma_e)_{e\in E(\G)}$ with each $\gamma_e\in Z(G_e)$ defines a Dehn twist. To see this, we have to verify the compatability condition \eqref{hed2} on page \pageref{hed2}. As $D$ has trivial vertex group automorphisms and $\gamma_e \in Z(G_e)$, we have:
$$H_{\tau(e)}(f_e(a))=f_e(a)=f_e(\gamma_e)f_e(a)f_e(\gamma_e)^{-1}=\delta(e)f_e(a)\delta(e)^{-1},$$
for any $a \in G_e$, as required. We say that the element $z_e=\g_e\g_{\bar{e}}^{-1}$ is the \emph{twistor} of the edge $e$. It is easy to verify that Dehn twists form a subgroup of $\aut(\mathcal{G})$.

\begin{definition}
Let $G$ be any group. An element $\phi\in\aut(G)$ or $\out(G)$ is a Dehn twist automorphism if there exists a graph of groups $\mathcal{G}$ and an isomorphism $\rho:G\to\pi_1(\mathcal{G},v)$ such that $\rho\phi\rho^{-1}$ is represented by a Dehn twist on $\mathcal{G}$.
\end{definition}

\begin{remark}\label{rm:dteq}
Our definition of Dehn twist here coincides with the notion of Dehn twist in \cite{CL} defined by a set of twistors $(z_e)_{e \in E}$ such that each $z_e \in Z(G_e)$ and $z_e=z_{\bar{e}}^{-1}$. Conversely, if we are given a set of twistors $(z_e)_{e\in E(\G)}$ such that each $z_e\in Z(G_e)$ and $z_{\bar e}=z_e^{-1}$,  we may take an orientation $E^+$ of $\G$ (a subset of $E(\G)$ such that for every edge $e$ exactly one element of $\{e,\bar{e}\}$ lies in $E^+$), and define a Dehn twist in our sense by taking
$$\gamma_e=\begin{cases}
z_e^{-1},&\text{if }e\in E^+,\\
1,&\text{if }e\notin E^+.
\end{cases}$$
\end{remark}

\subsection{The subgroup of Dehn twists in $\aut(\mathcal{G})$.}\label{s:dt2}

\begin{definition}Let $DA(\mathcal{G})$ and $DO(\mathcal{G})$ be the images of the subgroup of Dehn twists in $\aut(\mathcal{G})$ in $\aut(\pi_1(\mathcal{G},v))$ and $\out(\pi_1(\mathcal{G},v))$ respectively.\end{definition}

To look at these groups, we recall the following proposition from \cite{CL}:

\begin{proposition}[\cite{CL}, Proposition 5.4] \label{p:dt} Let $\mathcal{G}$ be a graph of groups with the property:\\
$(*)$ for every edge $e \in E(\G)$ there is an element $r_e \in G_{\tau(e)}$ with $$f_e(G_e) \cap r_ef_e(G_e)r_e^{-1} = 1.$$ Then two Dehn twists $D=(\mathcal{G},(\g_e)_{e \in E(\G)})$ and $D'=(\mathcal{G}, (\g_e')_{e \in E(\G)})$ with twistors $z_e=\g_e\g_{\bar{e}}^{-1}$ and $z_e'=\g_e'\g^{'-1}_{\bar{e}}$ determine the same outer automorphism of $\pi_1(\mathcal{G},v)$ if and only if $z_e=z'_e$ for all $e \in E(\G)$. 
\end{proposition}

If each edge group $G_e \cong \mathbb{Z}$, then if we take an orientation $E^+$ of $E(\G)$ (as in Remark \ref{rm:dteq}), we have $\mathbb{Z}^{|E^+|}$ choices for Dehn twists on $\mathcal{G}$ with distinct image in $\out(\pi_1(\mathcal{G},v))$. Multiplication of two Dehn twists is given by multiplying the twistors on each edge. Therefore Proposition~\ref{p:dt} implies:

\begin{proposition}\label{p:dt2} Let $\mathcal{G}$ be a graph of groups with property $(*)$ such that each edge group $G_e \cong \mathbb{Z}$. Then $DA(\mathcal{G})$ and $DO(\mathcal{G})$ are free abelian groups of rank equal to the number of geometric (or unoriented) edges of $\G$, i.e. the size of an orientation of $E(\G)$. \qed
\end{proposition}

Condition $(*)$ is satisfied if $f_e(G_e)$ is a proper malnormal subgroup of $G_{\tau(e)}$ for all $e$ (a subgroup $H$ of a group $G$ is malnormal if $gHg^{-1}\cap H = 1$ unless $g \in H$). With some further restrictions, the kernel of the map $\bar{A}$ defined in Proposition \ref{p:BJ2} can be given in terms of Dehn twists.

\begin{theorem}\label{t:vertex}
Suppose that the action of $\pi_1(\mathcal{G},v)$ on the Bass--Serre tree $T_\mathcal{G}$ is minimal and non-abelian. Further  suppose that the centre of $\pi_1(\mathcal{G},v)$ is trivial and the image of each edge group in an adjacent vertex group is malnormal. Then the kernel of the map $$\bar{A} \co \im U^0 \to \bigoplus_{w\in V(\G)}\out(G_w)$$ given in Proposition \ref{p:BJ2} is equal to $DO(\mathcal{G})$.
\end{theorem}

\begin{proof}
We give a brief description of how this is implied by Theorem 8.1 of \cite{BJ}. Bass and Jiang  use the notation $\ker \bar{A}=H^{(V)}$ and describe this group in terms of a chain  of normal subrgoups: $$H^{(V)} \rhd H^{(V,E)} \rhd H^{(V,E]} \rhd H^{(V,EZ]}.$$ The quotient $H^{(V)}/H^{(V,E)}$ is described by non-inner automorphisms of $G_e$ induced by conjugation by its normalisers in the two adjacent vertex groups (part (5) of \cite[Theorem 8.1]{BJ}). As $G_e$ is malnormal in each vertex group its image is equal to its normaliser and the quotient $H^{(V)}/H^{(V,E)}$ is trivial. The quotient $H^{(V,E)}/H^{(V,E]}$ is trivial if the centraliser of $f_e(G_e)$ in $G_{\tau(e)}$ is equal to the centre of $f_e(G_e)$ (part (6) of \cite[Theorem 8.1]{BJ}). This is also implied by the malnormality of $f_e(G_e)$ in $G_{\tau(e)}$.
Finally, if the centre of $\pi_1(\mathcal{G},v)$ is trivial, then the subgroup $H^{(V,EZ]}$ is trivial (part (8) of \cite[Theorem 8.1]{BJ}).  The subgroup $H^{(V,E]}$ defined in Section 7.4 of \cite{BJ} coincides with $DO(\mathcal{G})$. 
\end{proof}

We will be looking at a particular class of graphs of groups which satisfy all the above hypotheses.

\section{Centralisers of efficient Dehn twists in $\outn$}\label{s:cdt}

Cohen and Lustig give a notion of when a Dehn twist is \emph{efficient}. This may be thought of as when the graph of groups $\mathcal{G}$ that the Dehn twist is defined on is, in a certain sense, optimal.

\begin{definition}\label{d:bonding}
Let $D$ be the Dehn twist given by $(\mathcal{G},(\gamma_e)_{e\in E(\G)})$ with twistors $z_e=\g_e\g_{\bar e}^{-1}$. Two edges $e'$ and $e''$ with common terminal vertex $w$ are called
\begin{itemize}
\item {\em positively bonded}, if there exist $m,n\ge1$ such that $f_{e'}(z_{e'}^m)$ and $f_{e''}(z_{e''}^n)$ are conjugate in $G_w$,
\item {\em negatively bonded}, if there exist $m\ge1$ and $n\le-1$ such that $f_{e'}(z_{e'}^m)$ and $f_{e''}(z_{e''}^n)$ are conjugate in $G_w$.
\end{itemize}
\end{definition}

\begin{definition}[cf. Definition 6.2 in \cite{CL}]\label{efficient}
A Dehn twist $D$ given by $(\mathcal{G},(\gamma_e)_{e\in E(\G)})$ is called \emph{efficient} if:
\begin{enumerate}
\item $\mathcal{G}$ is \emph{minimal:} If $w$ has valence one and $w=\tau(e)$, then the edge map $f_e\co G_e\rightarrow G_w$ is not surjective.
\item There is \emph{no invisible vertex:} There is no 2-valent vertex $w$ such that both edge maps $f_{e_i}\co G_{e_i}\rightarrow G_w$ are surjective, where $\tau(e_1)=\tau(e_2)=w$ and $e_1\neq e_2$.
\item \emph{No unused edges:} For every edge $e$, we have $z_e \neq 0$, or equivalently $\gamma_{\bar e}\neq\gamma_e$.
\item \emph{No proper powers:} If $r^p\in f_e(G_e)$ for some $p\neq0$, then $r\in f_e(G_e)$.
\item Whenever $w=\tau(e_1)=\tau(e_2)$, then $e_1$ and $e_2$ are not \emph{positively bonded}.
\end{enumerate}
\end{definition}

Suppose that $\pi_1(\mathcal{G},v)\cong F_n$ and $D$ is an efficient Dehn twist on $\mathcal{G}$. The above conditions imply that each edge group is infinite cyclic and each vertex group is free of rank at least two (see Proposition 6.4 of \cite{CL} for more detail). In particular, no edge map of $\mathcal{G}$ is surjective. The action on the Bass-Serre tree $T_\mathcal{G}$ is minimal and non-abelian, and $\mathcal{G}$ satisfies condition $(*)$ of Proposition~\ref{p:dt}.

\subsection{Representing centralisers by abstract automorphisms}

We shall now restrict ourselves to the case where $\pi_1(\mathcal{G},v)$ is a free group of rank $n\ge2$, and $D=(\mathcal{G},(\gamma_e)_{e\in\G})$ is an efficient Dehn twist defined on $\mathcal{G}$ (we shall later see that every Dehn twist automorphism has an efficient representative). We fix an isomorphism $F_n\cong\pi_1(\mathcal{G},v)$, so that $D_{*v}$ and $\widehat{D}$ are identified with elements of $\aut(F_n)$ and $\out(F_n)$ respectively.

The action of $F_n$ on the Bass--Serre tree $T_\mathcal{G}$ of $\mathcal{G}$ is minimal and very small, and so by endowing the edges of $\mathcal{G}$ with varying lengths, it defines a simplex fixed by $\widehat{D}$ in the boundary of Outer space (this is described in detail in \cite{CL2}). Cohen and Lustig proved that this simplex contains all simplicial actions fixed by $\widehat{D}$ in the boundary of Outer space, and use this to show that elements of $C(\widehat{D})$ may be represented by elements of  $\aut(\mathcal{G})$. It is not true that the image of every element of $\aut(\mathcal{G})$ in $\outn$ lies in $C(\widehat{D})$. For this reason we need to pass to a subgroup:

\begin{definition}
We define $$\aut(\mathcal{G},\widehat{D})=\{H \in \aut(\mathcal{G}) :H_e(z_e)=z_{H(e)} \text{ for all } e \in E(\G) \} . $$ This is the subgroup of $\aut(\mathcal{G})$ consisting of elements $H$ which preserve the twistors of $D$. We will also need to look at the subgroup of this group preserving the basepoint, which we define by: $$\aut(\mathcal{G},D_{*v})=\{H \in \aut(\mathcal{G},\widehat{D}): H_\G(v)=v \}.$$ 
\end{definition}

Proposition 7.1 of \cite{CL} may be rephrased as follows:

\begin{proposition}\label{p:autgd}
Let $D=(\mathcal{G},(\g_e)_{e \in E(\G)})$ be an efficient Dehn twist. The maps $H \mapsto \widehat{H}$ and $H\mapsto H_{*v}$ induce surjective homomorphisms
\begin{align*}\outgd&\to C(\widehat{D}),\\
\aut(\mathcal G, D_{*v})&\to C(D_{*v}).
\end{align*}\end{proposition}

This allows us to describe the centralisers of efficient Dehn twists by the Bass--Jiang decomposition of the subgroup of $\out(F_n)$ preserving $\mathcal{G}$.

\subsection{A short exact sequence for $C^0(\widehat{D})$}

We define: $$\aut^0(\mathcal{G},\widehat{D})=\aut(\mathcal{G},\widehat{D})\cap \aut^0(\mathcal{G}).$$ This is the subgroup of Dehn twists which act trivially on $\G$ and preserve the twistors of $\widehat{D}$. We denote by $C^0(\widehat{D})$ the image of $\aut^0(\mathcal{G},\widehat{D})$ in $C(\widehat{D})$. By Proposition \ref{p:graph} the homomorphism $$\alpha:\aut(\mathcal{G},\widehat{D}) \to \aut(\G)$$  given by $H \mapsto H_\G$ descends to a homomorphism $$\bar \alpha: C(\widehat{D}) \to \aut(\G)$$ with kernel $C^0(\widehat{D})$. Also, by Proposition~\ref{p:BJ2} the homomorphism $$A:\aut^0(\mathcal{G},\widehat{D}) \to \bigoplus_{w\in V(\G)} \out(G_w)$$ given by $H \mapsto (H_w)_{w \in V(\G)}$ descends to a homomorphism $$\bar{A}:C^0(\widehat{D}) \to \bigoplus_{w\in V(\G)} \out(G_w).$$  All Dehn twists in $\aut(\mathcal{G})$ lie in $\aut^0(\mathcal{G},\widehat{D})$ so by Theorem~\ref{t:vertex} the kernel of $\bar{A}$ is the whole of $DO(\mathcal{G})$. To complete our decomposition of $C(\widehat{D})$ we only need to describe the image of $\bar A$.

\begin{definition}\label{d:conjclass} We  pick a generator $a_e$ of each edge group $G_e$. Let $\mathcal{C}_w$ be the set of conjugacy classes in $G_w$ defined by: $$\mathcal{C}_w=\{[f_e(a_e)]:e \in E(\G), \tau(e)=w\}.$$
Let $\aut(G_w,{\C_w})$ and $\out(G_w,{\C_w})$ be the subgroups of $\aut(G_w)$ and $\out(G_w)$ respectively consisting of automorphisms that fix every conjugacy class in the finite set $\mathcal{C}_w$.\end{definition}

\begin{lemma}\label{l:dt} The image of the map $\bar{A} \co C^0(\widehat{D}) \to \bigoplus_{w \in V(\G)}\out(G_w)$ (equivalently, the map $A \co \aut^0(\mathcal{G},\widehat{D}) \to \bigoplus_{w \in V(\G)}\out(G_w)$) is equal to $\bigoplus_{w \in V(\G)} \out(G_w,\mathcal{C}_w)$. \end{lemma}

\begin{proof} Let $H \in \aut^0(\mathcal{G},\widehat{D})$. As $H_\G=1$ and $H$ preserves twistors, we have $H_e(z_e)=z_e$ for any edge. As each twistor is nontrivial and each edge group is cyclic this implies that $H_e=1$. The consistency condition for elements of $\aut(\mathcal{G})$  (equation \eqref{hed2} on page \pageref{hed2}) then implies: \begin{equation*}  H_{\tau(e)}(f_e(a))=\delta(e)f_e(a)\delta(e)^{-1}. \end{equation*} Applying this equation over all edges with $\tau(e)=w$ shows that $H_w$ fixes every conjugacy class in $\mathcal{C}_w$, so lies in $\aut(G_w,\C_w)$.

It only remains to show that $A$ is surjective. To do this, we take an element $H_w \in \aut(G_w,\mathcal{C}_w)$ for each vertex of $\G$ and build an element $H \in \autog$ with $H_w$ as the vertex automorphism at $w$. As $H_w \in \aut(G_w,\mathcal{C}_w)$, the conjugacy class of $f_e(a_e)$ is preserved by $H_w$, so there exists $\delta(e) \in G_w$ such that $H_w(f_e(a_e))=\delta(e)f_e(a_e)\delta(e)^{-1}$. As $a_e$ generates $G_e$, this identity holds for every element $a \in G_e$, and we may define $H$ to be the automorphism of $\mathcal{G}$ given by the trivial graph automorphism $H_\G=1$, the vertex automorphisms $(H_w)_{w \in V(\G)}$, 
trivial edge automorphisms, and twisting factors $(\delta(e))_{e \in E(\G)}$.
\end{proof}

\begin{remark}
The above lemma is true for the map $A$ in the more general situation that $\widehat{D}$ has nontrivial twistors and cyclic edge groups. However, efficiency is required for the correspondence between $C^0(\widehat{D})$ and $\aut(\mathcal{G},\widehat{D})$, and therefore the description of the map $\bar{A}$.
\end{remark}

Assembling the work in this section provides our main theorem for centralisers of Dehn twist automorphisms in $\out(F_n)$:

\begin{theorem}\label{t:ses}
Let $D$ be an efficient Dehn twist on a graph of groups $\mathcal{G}$ with $\pi_1(\mathcal{G},v)\cong F_n$. Let $C(\widehat{D})$ be the centraliser of $D$ in $\outn$. There exists a homomorphism: \begin{align*} \bar{\alpha}\co &C(\widehat{D}) \to \aut(\G),\end{align*} with kernel $C^0(\widehat{D})$ a finite index subgroup fitting into the exact sequence: \begin{gather*}1\to DO(\mathcal{G}) \to C^0(\widehat{D}) \to \bigoplus_{w \in V(\G)} \out(G_w,\mathcal{C}_w)\to1,  \end{gather*}
where $DO(\mathcal{G})$ is a free abelian group of Dehn twists of rank equal to the number of geometric edges of $\mathcal{G}$. \qed \end{theorem}

\subsection{Finiteness properties}\label{s:finprop}

Every Dehn twist in $\out(F_n)$ may be represented by an efficient Dehn twist $D$ (\cite{CL}, Section 8.2.). Each group $\out(G_w,\mathcal{C}_w)$ in the exact sequence for $C^0(\widehat{D})$ has a finite-index torsion-free subgroup $A_w$ with finite classifying space $K(A_w,1)$ (\cite{CV}, Corollary~6.1.4.). Let $H$ be the intersection of the preimages of $A_w$ in $C^0(\widehat{D})$. Then $H$ fits in the exact sequence
$$1 \to DO(\mathcal{G}) \to H \to \bigoplus_{w \in V(\G)} A_w \to 1,$$
and, as both ends of this exact sequence have finite classifying spaces, so does $H$ (see, for example, Theorem 7.1.10 of \cite{G}).

\begin{corollary}\label{c:fpout} If $\phi \in \out(F_n)$ is a Dehn twist automorphism then $C(\phi)$ has a finite-index, torsion-free subgroup with finite classifying space. \qed \end{corollary}

In particular, the centraliser of a Dehn twist automorphism in $\out(F_n)$ is finitely presented. In Section \ref{s:alg} we prove that the process of finding a presentation for the centraliser of a Dehn twist in $\aut(F_n)$ or $\out(F_n)$ can be made algorithmic.

\begin{remark}
The centraliser $C(\widehat{D})$ is a finite index subgroup of the stabiliser of the simplex corresponding to the tree $T_\mathcal{G}$ in the boundary of Outer space, and it follows that the stabilisers of these simplices satisfy the same finiteness properties. More generally, Guirardel and Levitt show that the group of automorphisms preserving any splitting of $F_n$ with cyclic subgroups is of type VF, and generalise this to toral relatively hyperbolic groups \cite{GL2}.\end{remark}

\section{Pointedly efficient Dehn twists}\label{s:peff}

It is shown in Section 8 of \cite{CL} that any Dehn twist automorphism in $\out(F_n)$ is represented by an efficient Dehn twist. This is not the case with $\aut(F_n)$. In this section we introduce the notion of a \emph{pointedly efficient} Dehn twist, and show that every Dehn twist automorphism in $\aut(F_n)$ is represented by a pointedly efficient Dehn twist.  We show that the centraliser of a pointedly efficient Dehn twist automorphism in $\aut(F_n)$ has a similar decomposition to the centraliser of a Dehn twist automorphism in $\out(F_n)$.

The definition of being pointedly efficient comes by treating the basepoint of the graph of groups differently, and is natural when working with automorphisms rather than the outer automorphism class they represent. For instance, every automorphism of linear growth in $\aut(F_n)$ is a root of a pointedly efficient Dehn twist (this follows from Propositions 4.25, 7.22, and Theorem 8.6 of \cite{R}).

\subsection{Building efficient representatives}

In \cite{CL}, Cohen and Lustig describe modifications of a Dehn twist $D$ and its underlying graph of groups $\mathcal{G}$. We briefly describe them here and refer the reader to Definition 8.2 of \cite{CL} for the precise definitions (\emph{positively} and \emph{negatively bonded} edges are described in Definition~\ref{d:bonding}).

\begin{enumerate}[(M1)]
\item {\em Transition to a proper subgraph:} Remove a valence one vertex $w$ when its corresponding edge map is surjective.
\item {\em Delete an invisible vertex with negatively bonded edges:} Remove  a valence two vertex $w$ when both edge maps are surjective and negatively bonded.
\item {\em Fold positively bonded edges:} Fold two positively bonded edges $e$ and $e'$ at a vertex $w$.
\item {\em Contract unused edges:} Collapse an edge $e$ with trivial twistor and replace the vertex group(s) with an HNN extension or an amalgam, depending on whether $e$ was a loop or not.
\item {\em Get rid of proper powers:} Adjoin a formal root to an edge group when $f_e(a_e)$ is a proper power. 
\end{enumerate}
If $D$ is a Dehn twist on $\mathcal{G}$ that fails (1), (2), (3), (4), or (5) in the definition of efficient (Definition \ref{efficient}) then we may apply one of the moves (M1)--(M5) to obtain a new Dehn twist $D'$ on a graph of groups $\mathcal{G}'$. Furthermore:

\begin{lemma}[Lemma 8.3, \cite{CL}] For any of the operations (M1)--(M5) and any vertices $w \in V(\G)$ and $w' \in V(\G')$ there exists an isomorphism $\rho: \pi_1(\mathcal{G},w) \to \pi_1(\mathcal{G}',w')$ such that $\widehat{D}'=\rho\widehat{D}\rho^{-1}$.\end{lemma}

In the $\aut(F_n)$ case, if we perform a move (M1) or (M2) when the vertex removed is equal to our chosen basepoint $v$, this may cause problems (we cannot choose an arbitrary basepoint). For this reason we define:

\begin{enumerate}[(M1*)]
\item Perform (M1) at a vertex $w\neq v$.
\item Perform (M2) at a vertex $w\neq v$.
\end{enumerate}

\begin{lemma} Let $D$ be a Dehn twist on a graph of groups $\mathcal{G}$ with basepoint $v$. For any of the operations (M1*), (M2*), (M3)--(M5) there exists a vertex $v' \in V(\G')$ and an isomorphism $\rho: \pi_1(\mathcal{G},v) \to \pi_1(\mathcal{G}',v')$ such that $D_{*v}'=\rho D_{*v}\rho^{-1}$.\end{lemma}

\begin{proof}
This follows from the proof of Lemma 8.3 of \cite{CL}. The isomorphism $\rho$ may be chosen so that $D'$ represents the same element of $\aut(F_n)$ (rather than $\out(F_n)$) as long as we do not remove the base vertex using  moves (M1) or (M2).
\end{proof}

\subsection{Pointedly efficient Dehn twists}

The work above tells us that we need to be more careful when improving representatives of Dehn twist automorphisms in $\aut(F_n)$. This leads to the following definition:

\begin{definition}
Let $D$ be a Dehn twist on a graph of groups $\mathcal{G}$ with a chosen basepoint $v$. It is called {\em pointedly efficient} if
\begin{enumerate}[(1*)]
\item $\mathcal G$ is {\em minimal away from the basepoint:} if $w\neq v$ has valence one and $w=\tau(e)$, then the edge map $f_e$ is not surjective,
\item there is {\em no invisible vertex away from the basepoint:} There is no 2-valent vertex $w\neq v$ such that both edge maps $f_{e_i}$ are surjective, where $\tau(e_1)=\tau(e_2)=w$ and $e_1\neq e_2$,
\end{enumerate}
and the conditions (3)--(5) of Definition \ref{efficient} are satisfied.
\end{definition}

Proposition 8.4 of \cite{CL} tells us that for any Dehn twist one can iteratively apply the operations (M1)--(M5) only a finite number of times.  A Dehn twist is pointedly efficient if we are unable to apply any of the moves (M1*), (M2*), (M3)--(M5). The process of checking for (pointed) efficiency and applying the above moves can be made algorithmic (this is described in detail in \cite{CL}). Hence: 

\begin{proposition}\label{p:peff_rep}Iteratively applying moves (M1*), (M2*), (M3)--(M5) gives an algorithm to obtain a pointedly efficient representative from any Dehn twist representative of an element of $\aut(F_n)$. \qed \end{proposition}

In order to use the results of Bass and Jiang, we require that the graph of groups $\mathcal{G}$ is minimal. This is not always the case for a pointedly efficient Dehn twist, as edge maps can be surjective at the base vertex $v$. We get around this by \emph{stabilisation}: we replace $G_v$ with $G_v \ast \mathbb{Z}$ so that the new graph of groups $\mathcal{G}'$ is now minimal.

\subsection{Stabilisation homomorphisms}

Given a graph of groups $\mathcal{G}$ with chosen basepoint $v$, the \emph{stabilisation} $\mathcal{G}'$ of $\mathcal{G}$ is the graph of groups obtained as follows: The underlying graph $\G$ is the same. $G_v$ is replaced by $G_v'=G_v\ast \mathbb{Z}$, whereas the other vertex groups and all edge groups are not modified. There is an obvious inclusion $i \co G_v \to G_v'$. We define $f_e'=f_e$ if $\tau(e) \neq v$, otherwise define $f_e'=i \circ f_e$. The injection $i$ then induces stabilisation maps:
\begin{align*} i_{G_v}\co &\aut(G_v) \to \aut(G'_v), \\  i_{F_n}\co &\aut(\pi_1(\mathcal{G},v)) \to\aut(\pi_1(\mathcal{G'},v))\cong\aut(F_{n+1}), \\ i_{\mathcal{G}}\co &\aut(\mathcal{G},v)\to\aut(\mathcal{G}',v).
\end{align*}

The first two maps are defined by extending the relevant automorphism to act trivially on the new free factor. $i_{\mathcal{G}}(H)=H'$ is defined to be the automorphism of $\mathcal{G}'$ with $H_v'=i_{G_v}(H_v)$ and $\delta'(e)=i(\delta(e))$ if $\tau(H_\G(e))=v$. The remaining data is the same as for $H$.

Given a Dehn twist $D=(\mathcal{G},(\gamma_e)_{e\in E(\G)})$ of $\mathcal{G}$, the same elements $\gamma_e$ define a Dehn twist $D'$ on $\mathcal{G}'$, which we refer to as the {\em stabilisation} of $D$. Then $i_{\mathcal{G}}(D)=D'$ and $i_{F_n}(D_{*v})=D_{*v}'$. One can check the following:

\begin{lemma}\label{l:peffst}
$D$ is pointedly efficient if and only if its stabilisation $D'$ is efficient. \qed
\end{lemma}

We have the following commutative diagram:
\begin{equation}\label{cd}  \xymatrix{\autgd\ar[r]\ar[d]^{i_{\mathcal{G}}}&C(D_{*v})\ar[d]^{i_{F_n}}\\
\aut(\mathcal{G}',D'_{*v})\ar[r]&C(D'_{*v})} \end{equation}
where the vertical maps are given by stabilisation and the horizontal maps by $H\mapsto H_{*v}$. The vertical maps are clearly injective, which we can use to extend the results from Section~\ref{s:background}.

\begin{prop}\label{p:image}
Given $H'\in\aut(\mathcal G',v)$, we have $H'=i_\mathcal G(H)$ for some $H\in\aut(\mathcal G,v)$ if and only if $H_v'$ lies in the image of $i_{G_v}$. If additionally $H'\in\aut(\mathcal G',D_{*v}')$ (or $H'$ is a Dehn twist), then $H\in\aut(\mathcal G,D_{*v})$ (or $H$ is a Dehn twist respectively).
\end{prop}

\begin{proof}
If $H'=i_\mathcal{G}(H)$ for some $H \in \aut(\mathcal{G},v)$ then $H_v'=i_{\mathcal{G}_v}(H_v)$ by the definition of the map $i_\mathcal{G}$. 
 
Conversely suppose that $H'_v=i_{G_v}(H_v)$ for some $H_v \in \aut(G_v)$. We want to find $H \in \aut(\mathcal{G},v)$ with $i_\mathcal{G}(H)=H'$. We can build such a $H$ using the rest of the defining data of $H'$ as long as $\delta'(e) \in G_v\subset G_v'$ for each edge $e$ with $\tau(e)=v$. Suppose we have such an edge $e$ with edge group generated by $a$. Then by the compatibility condition for automorphisms of graphs of groups:
\begin{align*} \delta'(e) f'_{H'(e)}(H_e(a)) \delta'(e)^{-1} & = H'_v(f'_e(a)) \\ &= i_{G_v}(H_v)(f'_e(a)). \end{align*}
As $f_e'=i \circ f_e$ and $f_{H'(e)}'=i \circ f_{H'(e)}$, where $f_e$ and $f_{H'(e)}$ are the edge maps in $\mathcal{G}$, the elements $i_{G_v}(H_v)(f'_e(a_e))$ and $f_{H'(e)}'(H'_e(a_e))$ lie in $G_v\smallsetminus\{1\}$. As $G_v$ is malnormal in $G_v'$ this implies $\delta'(e)$ lies in $G_v$ also.

This proves the first sentence of the proposition. The other assertions follow because the edge group homomorphisms of $H$ and $H'=i_\mathcal G(H)$ agree.
\end{proof}

\subsection{Representing centralisers by abstract automorphisms}

\begin{proposition}
Suppose that $D$ is a  pointedly efficient Dehn twist on a graph of groups $\mathcal{G}$. Then the map $$V: \aut(\mathcal{G},D_{*v}) \to C(D_{*v})$$ given by $H\mapsto H_{*v}$ is surjective.
\end{proposition}

\begin{proof}
Let $\phi \in C(D_{*v})$ and let $\phi'=i_{F_n}(\phi)$. By Proposition~\ref{p:autgd}, the lower horizontal map in the commutative square \eqref{cd} is surjective, so there exists $H' \in \aut(\mathcal{G'},D'_{*v})$ such that $H_{*v}'=\phi'$. Restricting to the vertex group $G'_v=G_v \ast \mathbb{Z}$ gives: $$H'_v = \phi'|_{G'_v} = i_{G_v}(\phi|_{G_v}).$$  Proposition~\ref{p:image} then tells us that $H'=i_\mathcal{G}(H)$ for some $H \in \aut(\mathcal{G},D_{*v})$, and by commutativity of the square \eqref{cd} we have $H_{*v}=\phi$.
\end{proof}

\begin{proposition} \label{p:id}
Suppose that $D$ is a  pointedly efficient Dehn twist on a graph of groups $\mathcal{G}$ and $H \in \aut(\mathcal{G},D_{*v})$ with $H_{*v}=1$. Then:
\begin{enumerate}
\item $H_\G=1$.
\item For each vertex $w \neq v$ there exists $g_w \in G_w$ such that $H_w=ad(g_w)$.
\item $H_v=1$.
\end{enumerate}
\end{proposition}

\begin{proof}
Following around the commutative square \eqref{cd} we see that $H_{*v}' = i_\mathcal{G}(H)_{*v}=i_{F_n}(H_{*v})=1$. As $T_{\mathcal{G}'}$ is minimal and non-abelian, Theorem~\ref{t:BJ} tells us that $H'_\G=1$ and $H'_w$ is inner for all $w \in V(\G)$. However $H'_\G=H_\G$ and $H'_w=H_w$ for all $w \neq v$. This proves parts (1) and (2). For part (3), we observe that $H_v$ is the restriction of $H_{*v}$ to the vertex group $G_v$, so must be the identity automorphism.
\end{proof}

\subsection{The pointed version of the exact sequence}

Part (1) of Proposition~\ref{p:id} implies that two different elements of $\aut(\mathcal{G},D_{*v})$ representing the same element of $C(D_{*v})$ must act on $\G$ in the same way. Hence:

\begin{proposition}\label{p:pgraph}
Let $D$ be a pointedly efficient Dehn twist on a graph of groups $\mathcal{G}$. Then the  homomorphism $$\beta: \aut(\mathcal{G},D_{*v}) \to \aut(\G,v)$$ given by $H \mapsto H_\G$ descends to a homomorphism \[\pushQED{\qed} \bar{\beta} : C(D_{*v}) \to \aut(\G,v).\qedhere\popQED\]
\end{proposition}

Let $\aut^0(\mathcal{G},D_{*v})= \ker \beta$ and let $C^0(D_{*v})= \ker \bar \beta$.

\begin{proposition}
Let $D$ be a  pointedly efficient Dehn twist on a graph of groups $\mathcal{G}$. Then there exists a surjective homomorphism $$B: \aut^0(\mathcal{G},D_{*v}) \to \aut(G_v,\mathcal{C}_v) \oplus \bigoplus_{w \neq v} \out(G_w,\mathcal{C}_w)$$ given by $H \mapsto (H_v, (\widehat{H_w})_{w \neq v})$  which descends to a homomorphism $$\bar{B} : C^0(D_{*v}) \to \aut(G_v,\mathcal{C}_v) \oplus \bigoplus_{w \neq v} \out(G_w,\mathcal{C}_w).$$
\end{proposition}

\begin{proof}
The mapping $H \mapsto (H_v, (\widehat{H_w})_{w \neq v})$ defines a homomorphism $$B\co \aut^0(\mathcal{G},D_{*v}) \to \aut(G_v) \oplus \bigoplus_{w\neq v} \out(G_w).$$ To show that the image of $B$ is the group $\aut(G_v,\mathcal{C}_v) \oplus \bigoplus_{w \neq v} \out(G_w,\mathcal{C}_w)$, one follows the proof of Lemma~\ref{l:dt}, and we shall omit the details. The map descends to $C^0(D_{*v})$ by parts (2) and (3) of Proposition~\ref{p:id}.
\end{proof}

Next we show that $DA(\mathcal G)$ is the same for pointedly efficient case:

\begin{prop}
If $\mathcal G$ admits a pointedly efficient Dehn twist, then $DA(\mathcal G)$ is free abelian of rank equal to the number of geometric edges of the underlying graph $\G$.
\end{prop}

\begin{proof}
If $D$ is pointedly efficient on $\mathcal G$, then its stabilisation $D'$ is efficient on $\mathcal G'$. Since $DA(\mathcal G')$ is free abelian of the desired rank by Proposition \ref{p:dt2}, we only have to show that the stabilisation homomorphism $i_{F_n}:\aut(\pi_1(\mathcal G,v))\to\aut(\pi_1(\mathcal G',v))$ induces an isomorphism $DA(\mathcal G)\cong DA(\mathcal G')$.

Whenever $H$ is a Dehn twist on $\mathcal G$, then $i_{F_n}(H_{*v})=(i_\mathcal G(H))_{*v}$ lies in $DA(\mathcal G')$ because $i_\mathcal G(H)$ is a Dehn twist. Therefore $i_{F_n}(DA(\mathcal G))\subset DA(\mathcal G')$.

Conversely, if $H'$ is a Dehn twist on $\mathcal G'$, then Proposition \ref{p:image} shows $H'=i_\mathcal G(H)$ for some Dehn twist $H$ on $\mathcal G$, and $H_{*v}'=i_{F_n}(H_{*v})$. Therefore the map $DA(\mathcal G)\to DA(\mathcal G')$ is surjective. Finally, it is injective because it is the restriction of the injection $i_{F_n}$.
\end{proof}

\begin{proposition}
The kernel of $\bar{B}$ is the group $DA(\mathcal G)$ of Dehn twists.
\end{proposition}

\begin{proof}
Since Dehn twists have trivial vertex group automorphisms, the inclusion $DA(\mathcal G)\subset\ker\bar B$ is clear. 

Let now $\phi\in C^0(D_{*v})$ satisfy $\bar B(\phi)=1$. We write $\phi'=i_{F_n}(\phi)$ and $\widehat{\phi'}$ for its outer automorphism class. By commutativity of the right hand square in the diagram
$$\xymatrix{
\aut^0(\mathcal{G},D_{*v})\ar@{->>}[r]\ar[d]&C^0(D_{*v})\ar@{->>}[r]^-{\bar B}\ar@{^(->}[d]^{i_{F_n}}&\aut(G_v,\C_v)\oplus\bigoplus_{w\neq v}\out(G_w,\C_w)\ar@{->}[d]^{i_{G_v}\oplus1}\\
\aut^0(\mathcal{G}',\widehat{D})\ar@{->>}[r]&C^0(\widehat{D'})\ar@{->>}[r]^-{\bar A'}&\bigoplus_{w \in V(\G)}\out(G'_w,\C_w)}$$
we have $\widehat{\phi'}\in\ker\bar A'$. As $D'$ is efficient, Theorem \ref{t:vertex} shows $\ker\bar A'=DO(\mathcal G')$. Therefore we have a Dehn twist $H'$ of $\mathcal G'$ such that $\widehat{\phi'}=\widehat{H'}$ and $\phi'=\ad_W\circ H_{*v}'$ for some $W\in\pi_1(\mathcal G',v)$. Since $\bar B(\phi)=1$, we have $\phi|_{G_v}=1$ and $\phi'|_{G_v'}=i_{G_v}(\phi|_{G_v})=1$. Since $H'$ is a Dehn twist, we also have $H_v'=1$, and we see that $\ad_W=1$ on $G_v'$. As $G_v'$ has rank at least two, we conclude $W=1$, so $i_{F_n}(\phi)=\phi'=H_{*v}'$. Proposition \ref{p:image} provides a Dehn twist $H$ of $\mathcal G$ such that $i_\mathcal G(H)=H'$. We now have
$$i_{F_n}(\phi)=\phi'=H_{*v}'=(i_\mathcal G(H))_{*v}=i_{F_n}(H_{*v}).$$
By injectivity of $i_{F_n}$ this implies $\phi=H_{*v}$, so $\phi\in DA(\mathcal G)$. Hence $\ker\bar B\subset DA(\mathcal G)$.
\end{proof}

Combining the propositions in this section gives our main theorem for Dehn twist automorphisms in $\aut(F_n)$:

\begin{theorem}\label{t:pses}
Let $D$ be a pointedly efficient Dehn twist on a graph of groups $\mathcal{G}$ with $\pi_1(\mathcal{G},v)\cong F_n$. Let $C(D_{*v})$ be the centraliser of $D$ in $\autn$. There exists a homomorphism
$$\bar{\beta}:C(D_{*v})\to\aut(\G,v)$$
with kernel $C^0(D_{*v})$ a finite index subgroup fitting into the exact sequence
$$1\to DA(\mathcal{G}) \to C^0(D_{*v}) \to \aut(G_v,\C_v)\oplus\bigoplus_{w\neq v}\out(G_w,\C_w) \to1,$$
where $DA(\mathcal{G})$ is a free abelian group of Dehn twists of rank equal to the number of geometric edges of $\mathcal{G}$. \qed
\end{theorem}

\begin{remark}
The same reasoning as Corollary \ref{c:fpout} can also be used to show that the centraliser of any Dehn twist automorphism $\phi\in\aut(F_n)$ is of type VF. To see this, we first represent $\phi$ by some pointedly efficient Dehn twist. Then we use the exact sequence $1\to G_v\to\aut(G_v,\C_v)\to\out(G_v,\C_v)\to1$ to find a finite index subgroup of $\aut(G_v,\C_v)$ with a finite classifying space. Then we proceed as in Section \ref{s:finprop}.
\end{remark}

\section{Computing finite presentations}\label{s:alg}

In this section we describe how we can algorithmically determine a presentation for the centraliser of a Dehn twist automorphism in $\aut(F_n)$ or $\out(F_n)$. We have a description of a centraliser in terms of short exact sequences, and will use these to build the presentation.

\subsection{Short exact sequences}\label{s:ses}
 For this section, we will suppose that
$$1\to A\xrightarrow{\iota} B\xrightarrow{\pi} C\to1$$ is a short exact sequence of groups. Furthermore, suppose that $A$ and $C$ have finite presentations $A = \langle X | R \rangle$ and $C= \langle Z | S \rangle$, so that:
$$1\to \langle X | R \rangle\xrightarrow{\iota} B\xrightarrow{\pi} \langle Z | S \rangle\to1.$$

Let $\iota(X)$ be the image of $X$ in $B$ and let $\tilde Z$ be a subset of $B$ mapped bijectively to $Z$ under $\pi$. We use $\tilde z$ to denote the element of $\tilde Z$ mapped to $z \in Z$ under $\pi$. In this section we will use $w$, or other lower case letters, to denote words in the free group $F(\iota(X) \sqcup \tilde Z)$ with the basis  $\iota(X) \sqcup \tilde Z$ and $[w]$ to denote the element of $B$ determined by the word $w$. There are three types of relations in $B$ that we will make use of.

\begin{itemize}
\item \textbf{Kernel relators.} Suppose that $r=x_1^{\epsilon_1} \cdots x_k^{\epsilon_k}$ is an element of $R$, with each $x_i \in X$ and $\epsilon_i = \pm 1$. Let $$\bar{r}=\iota(x_1)^{\epsilon_1} \cdots \iota(x_k)^{\epsilon_k}.$$ Then $[\bar{r}]=1$ and we say that $\bar{r}$ is a \emph{kernel relator} in $F(\iota(X) \sqcup \tilde Z)$.
\item \textbf{Lifted relators.} Suppose that $s=z_1^{\epsilon_1} \cdots z_k^{\epsilon_k}$ is an element of $S$, with each $z_i \in Z$ and $\epsilon_i = \pm 1$. Let $\tilde s = \tilde z_1^{\epsilon_1} \cdots \tilde z_k^{\epsilon_k}$. As $\pi([\tilde s])=1$, it follows that $[\tilde s] \in \iota(A)$, so there exists a word $w_s=\iota(x_1)^{\epsilon_1} \cdots \iota(x_l)^{\epsilon_l}$ in $F(\iota(X))$ such that $[w_s]=[\tilde s]$. We say that $$\{ w_s= \tilde{s} : s \in S\}$$ is a set of \emph{lifted relators} in $F(\iota(X) \sqcup \tilde Z)$.
\item \textbf{Conjugation relators.} If $x \in X$, $z \in Z$, and $\epsilon \in \{1, -1\}$, then $\tilde z^\epsilon \iota(x) \tilde z^{-\epsilon}$ is mapped to the identity element under $\pi$. As above, there then exists a word $w_{x,z,\epsilon}=\iota(x_1)^{\epsilon_1} \cdots \iota(x_l)^{\epsilon_l}$ such that $[w_{x,z,\epsilon}]=[\tilde z^\epsilon \iota(x) \tilde z^{-\epsilon}]$ in $B$. We say that $$\{ w_{x,z,\epsilon}= \tilde z^\epsilon \iota(x) \tilde z^{-\epsilon} : x \in X, z \in Z, \epsilon=\pm 1\}$$ is a set of \emph{conjugation relators} in $F(\iota(X) \sqcup \tilde Z)$. \qedhere
\end{itemize}

The following proposition is a well-known result from combinatorial group theory:

\begin{proposition}\label{l:sespres} The group $B$ has a finite presentation given by the generating set $\iota(X) \sqcup \tilde Z$ and sets of kernel, lifted, and conjugation relators. \qed \end{proposition}

Given presentations of $A$ and $C$, more information is needed in order to algorithmically find lifted and conjugation relators and obtain a presentation of $B$. Roughly speaking, one needs to be able to calculate effectively in $B$, have a description of the map $\iota$, and we need a way to lift the generating set $Z$ of $C$ up to $\tilde Z \subset B$.

\begin{proposition} \label{p:crit}
Given a short exact sequence $$1\to \langle X | R \rangle\xrightarrow{\iota} B\xrightarrow{\pi} \langle Z | S \rangle\to1,$$ suppose that:
\begin{enumerate}
\item The group $B$ is a subgroup of a finitely generated group $G$. The group $G$ has a finite generating set $\langle Y \rangle$ and we have a solution to word problem with this generating set.
\item We have a description of each element of $\iota(X)$ as a word in $F(Y)$, and for each $z \in Z$ there is an algorithm to find a word $w_z$ in $F(Y)$ such that $[w_z] \in B$ and $\pi([w_z])=z$. \end{enumerate}
Then there is an algorithm to find a presentation of $B$. 
\end{proposition}

\begin{proof} 
In this case, we take each lift $\tilde{z}$ to be given by $[w_z]$. For an algorithmic version of Proposition~\ref{l:sespres}, one needs a method to find left hand words of lifted relators and conjugation relators. This is possible as the word problem is solvable in the ambient group $G$. To find each lifted relator we take the element $\tilde{s}$, choose an appropriate ordering of $F(\iota(S))$, and test words  $w_s \in F(\iota(S))$ in order until we find one such that $[w_s]=[\tilde{s}]$. The same method applies to find conjugation relators.
\end{proof}

We have taken care to avoid the assumption that we are given a finite generating set for $B$, as in general this is something that will be obtained from the algorithm. In this paper, $B$ will always be a subgroup of $\aut(F_n)$ or $\out(F_n)$, so condition (1) holds. We will also have a detailed enough description of the maps involved in order to find elements as in (2). One final note: in practice, it is far more convenient to describe automorphisms in terms of how they act on a fixed generating set of $F_n$ rather than as products of elements of a finite generating set of $\aut(F_n)$, so the generating set $\langle Y \rangle$ of the ambient group will not be used explicitly in the work that follows.

\subsection{Other Ingredients}
We will also need a pair of useful procedures:
\begin{itemize}
\item\textbf{The McCool complex.} In \cite{MC},  McCool describes an algorithm to build a 2-dimensional finite CW complex whose fundamental group can be identified with $\out(F_n,\C)$ for any set $\C$ of conjugacy classes (as well as a version for $\aut(F_n,\mathcal{C})$). This clearly leads to an algorithm providing a finite presentation for such a group.
\item\textbf{Whitehead's algorithm.} (\cite{LS}, Proposition I.4.21.) If we have two sets of conjugacy  classes $\alpha_1,\ldots,\alpha_t$ and $\alpha_1',\ldots,\alpha_t'$ in $F_n$ then it is decidable whether there is an automorphism $\phi \in \aut(F_n)$ such that $\phi(\alpha_1)=\alpha_1',\phi(\alpha_2)=\alpha_2',\ldots,\phi(\alpha_t)=\alpha_t'$. One can find such an automorphism if it exists.
\end{itemize}

\subsection{An algorithm to find a presentation of a centraliser}\label{s:alg_steps}

Suppose that we are given a Dehn twist representative for some $\phi\in\aut(F_n)$ or $\out(F_n)$.  To obtain a finite presentation of $C(\phi)$, one proceeds as follows:

\textbf{Step 1.} Convert the Dehn twist to a (pointedly) efficient representative $D$ as in Section \ref{s:peff}. By Theorem~\ref{t:ses} (in $\out(F_n)$) or Theorem~\ref{t:pses} (in $\aut(F_n)$), there is a short exact sequence 
$$1\to DO(\mathcal{G}) \to C^0(\widehat{D}) \to \bigoplus_{w \in V(\G)} \out(G_w,\mathcal{C}_w)\to1$$ or
$$1\to DA(\mathcal{G}) \to C^0(D_{*v}) \to \aut(G_v,\C_v)\oplus\bigoplus_{w\neq v}\out(G_w,\C_w) \to1.$$

\textbf{Step 2.} Find a presentation of the right hand term in this first short exact sequence by building the McCool complex for each summand.

\textbf{Step 3.} Use the respective short exact sequence to determine a presentation of $C^0(\widehat{D})$ or $C^0(D_{*v})$ using the method in Section~\ref{s:ses}. Part (2) of Proposition~\ref{p:crit} is satisfied as we have an explicit description of elements in the left-hand group as Dehn twists, and for each element $(H_w)_{w \in V(\G)}$ in the right-hand term of the exact sequence one can build a graph of groups automorphism $H$ mapping to $(H_w)_{w \in V(\G)}$ as in the proof of Lemma~\ref{l:dt}, then compute the automorphism represented by $H$ to lift generators of the right-hand group to the centraliser.

\textbf{Step 4.} Now turn to the second short exact sequence: $$1\to C^0(\widehat{D})\to C(\widehat{D})\overset{\bar\alpha}{\to}\text{im }\bar\alpha \to 1$$
or$$1\to C^0(D_{*v})\to C(D_{*v})\overset{\bar\beta} \to \text{im }\bar\beta\to 1.$$ We need to find a presentation for the image of $\bar\alpha$ or $\bar\beta$ in $\aut(\G)$: one must find which graph automorphisms can be induced by an arbitrary element $H$ of $\aut(\mathcal{G},\widehat{D})$ or $\aut(\mathcal{G},D_{*v})$. For each edge $e$ with twistor $z_e$ fix a generator $a_e\in G_e\cong\mathbb{Z}$ and $n_e>0$ such that $z_e=a_e^{n(e)}$. Since $z_{\bar e}=z_e^{-1}$, we get $a_{\bar e}=a_e^{-1}$ and $n(\bar{e})=n(e)$ for every edge $e$. In the $\out$ case, a graph isomorphism $h\in\aut(\G)$ is in the image of $\bar\alpha$ if and only if $n_{h(e)}=n_e$ for all $e$, and for each vertex $w$ there is an isomorphism from $G_w$ to $G_{h(w)}$ mapping each conjugacy class $[f_e(a_e)]$ in $G_w$ to the conjugacy class $[f_{h(e)}(a_{h(e)})]$ in $G_{h(w)}$. To check this, one first checks whether the ranks of the free groups agree and, if so, one uses Whitehead's algorithm to test for the existence of an automorphism preserving the prescribed conjugacy classes. In the $\aut$ case, the image of $\bar\beta$ is determined in the same way, 
but 
only for graph automorphisms fixing the basepoint $v$.

We can now choose a finite presentation for this subgroup of $\aut(\G)$, e.g. take all group elements as generators and the obvious relators given by group multiplication.

\textbf{Step 5.}  Compute a presentation of the centraliser $C(\widehat{D})$ (or $C(D_{*v})$) using the second exact sequence and the presentation of $C^0(\widehat{D})$ (or $C^0(D_{*v})$) and the presentation of $\text{im } \bar{\alpha}$ (or $\text{im } \bar{\beta}$, respectively). 

The conditions for Proposition~\ref{p:crit} are satisfied as we found automorphisms of graphs of groups representing elements of $\text{im } \bar{\alpha}$ and $\text{im } \bar{\beta}$ in Step 4 and can use these to find pre-images of elements of the right-hand group in the centraliser.

\begin{remark}
Although this is an algorithm to compute an explicit finite presentation, the McCool complex usually has a huge number of cells. Hence it is hard to write down the resulting presentation of the centraliser by hand. Therefore it is desirable to simplify these presentations. For the stabiliser of conjugacy classes of basis elements this has been done in \cite{JW}, and the first author describes simplified presentations for stabilisers of more general elements in \cite{R}.
\end{remark}

\section{Centraliser of a Nielsen automorphism} \label{s:Nielsen}

In this section we apply the work from the rest of the paper to Nielsen automorphisms, which have particularly simple Dehn twist representatives. We give a presentation for the centraliser of a Nielsen automorphism, use this presentation to compute the abelianisation of the centraliser, and finally describe how this computation restricts actions of $\aut(F_n)$ on CAT(0) spaces. 

\subsection{Nielsen automorphisms of $F_n$ as Dehn twists}\label{s:nielsen_dt}

Let $\G$ be the graph with one vertex $v$ and one loop $e$ (that is two oriented edges $e$ and $\bar e$). We take $G_v$ to be a free group with basis $B,b,c_1,\ldots,c_{n-2}$. Let $G_e$ be infinite cyclic with generator $r$. The edge maps are defined by $f_e(r)=b$ and $f_{\bar e}(r)=B$.

The fundamental group $\pi_1(\mathcal{G},v)$ is the full path group $\Pi(\mathcal{G})$ here. It is generated by $a:=t_e$, $B$, $b$, and the $c_i$ subject to the relation $aba^{-1}=B$. In other words, it is the free group with basis $a,b,c_1,\ldots,c_{n-2}$. We define a Dehn twist $D$ by $\gamma_e=r^{-1}$ and $\gamma_{\bar e}=1$. Then $\delta(e)=f_e(\gamma_e)=b^{-1}$ and $\delta(\bar e)=f_{\bar e}(\gamma_{\bar e})=1$. It follows that $D_{*v}$ maps $a$ to $ab$ and fixes $b$ and all $c_i$. Hence $\rho:=D_{*v}$ is a Nielsen automorphism. Note that $D$ is efficient. (See Sections \ref{s:action}, \ref{s:dtdef}, and \ref{s:cdt} for definitions and terminology.)

We may now compute an explicit presentation of the centraliser of $\rho$ by the algorithm outlined in Section \ref{s:alg_steps}. As $DA(\mathcal{G})$ is infinite cyclic and generated by $\rho$, the short exact sequence for $C^0(\rho)$ simplifies to
\begin{equation}\label{sesn}
1\to\langle\rho\rangle\to C^0(\rho)\to\aut(G_v,\C_v)\to1,
\end{equation}
where $\C_v=\{[b],[B]\}$ is a set of two conjugacy classes of basis elements. Hence the first item we need is a presentation of $\aut(G_v,\C_v)$. This was found in \cite{JW}, and we will review it in Proposition \ref{p:JW} below.

\subsection{A presentation for $\aut(G_v,\C_v)$}

In the following, we use $P_{i,j}$ to denote the automorphism of either $F_n=\langle a,b,c_1,\ldots,c_{n-2}\rangle$ or $G_v=\langle B,b,c_1,\ldots,c_{n-2}\rangle$ which permutes the basis elements $c_i$ and $c_j$. Similarly $I_i$ denotes the automorphism mapping $c_i$ to $c_i^{-1}$ and fixing the other basis elements.

If $y$ and $z$ are elements of a fixed basis of a free group and $\epsilon\in\{\pm1\}$, then $(y^\epsilon;z)$ is the automorphism fixing all basis elements different from $y$ and sending $y$ to $yz$ if $\epsilon=1$, and to $z^{-1}y$ if $\epsilon=-1$. Moreover $(y^\pm;z)$ is the partial conjugation $y\mapsto z^{-1}yz$ fixing the other basis elements.

We warn the reader that $(y^\pm;z)$ is {\em not} an abbreviation for $(y;z)$ or $(y^{-1};z)$, instead it denotes the composition of those two. The partial conjugation $(y^\pm;z)$ is usually used in a generating set when $(y;z)$ and $(y^{-1};z)$ are not elements of the given subgroup but the partial conjugation is.

The following is the special case $k=2$ of Proposition 7.1 in \cite{JW}. Let $y_1:=B=aba^{-1}$ and $y_2:=b$.

\begin{proposition}[\cite{JW}, Proposition~7.1] \label{p:JW}
The group $\aut(G_v,\C_v)$ is generated by the following elements:\\

\begin{tabbing}
$P_{i,j}$\qquad\=for $1\le i,j\le n-2$ and $i\neq j$,\\
$I_i$\>for $1\le i\le n-2$,\\
$(c_i^\epsilon;z)$\>for $1\le i\le n-2$, $\epsilon=\pm1$ and $c_i\neq z\in\{c_1,\ldots,c_{n-2},y_1,y_2\}$,\\
$(y_i^{\pm};z)$\>for $i\in\{1,2\}$ and $y_i\neq z\in\{c_1,\ldots,c_{n-2},y_1,y_2\}$.
\end{tabbing}
For all possible choices $z,z_i\in\{c_1,\ldots,c_{n-2},y_1,y_2\}$ and $w,w_i=c_{j_i}^{\delta_i}$ or $y_{j_i}^\pm$, whenever the symbols involved give well-defined generators or inverses from the list above, the following list is a collection of defining relators:
\begin{tabbing} 
Q1\qquad\=Relators in $\aut(F_{n-2})$ for $\{(c_i^\epsilon;c_j),P_{i,j},I_j\}$,\\
Q2\>$(w_1;z_1)(w_2;z_2)=(w_2;z_2)(w_1;z_1)$ for $w_1\neq w_2$ and $z_i^{\pm1}\notin\{w_1,w_2\}$,\\
Q3.1\>$(y_i^\pm;c_j)P_{j,l}=P_{j,l}(y_i^\pm;c_l)$,\\
Q3.2\>$(y_i^\pm;c_j)I_j=I_j(y_i^\pm;c_j^{-1})$,\\
Q3.3\>$P_{j,l}$, $I_j$ commute with $(y_1^\pm;y_2)$, $(y_2^\pm;y_1)$,\\
Q3.4\>$(c_j^\epsilon;y_i)P_{j,l}=P_{j,l}(c_l^\epsilon;y_i)$,\\
Q3.5\>$(c_j;y_i)I_j=I_j(c_j^{-1};y_i)$,\\
Q4.1\>$(w;c_j^{-\eta})(c_j^\eta;z)(w;c_j^\eta)=(w;z)(c_j^\eta;z)$,\\
Q4.2\>$(y_i^\pm;z^{-\epsilon})(w;y_i)(y_i^\pm;z^\epsilon)=(w;z^\epsilon)(w;y_i)(w;z^{-\epsilon})$,\\
Q5\>$(c_j^{-\eta};y_i)(y_i^\pm;c_j^\eta)=(y_i^\pm;c_j^\eta)(c_j^\eta;y_i^{-1})$.
\end{tabbing}
\end{proposition}

Here we read compositions from right to left. We warn the reader that the articles \cite{JW} and \cite{MC} use the opposite convention.

In the original article \cite{JW} the relators Q3.3 through Q3.5 are missing. It was independently noticed by Andrew Putman and by the first author that these relators have to be added.

\subsection{The presentation for $C(\rho)$}

By definition, $C^0(\rho)$ is the kernel of the map $\bar\alpha$ in Proposition \ref{p:pgraph}. Since $\G$ is a loop with a single vertex, $\aut(\G)=\aut(\G,v)$ is cyclic of order 2 generated by the isometry swapping $e$ with $\bar{e}$. The map $\bar{\beta}$ is surjective: Indeed, for $\theta\in C(\rho)$ defined by
$$\theta(x)=\begin{cases}a^{-1} ,&\text{if }x=a,\\ab^{-1}a^{-1} ,&\text{if }x=b,\\c_i ,&\text{if }x=c_i,\end{cases}$$
the graph automorphism $\bar{\beta}(\theta)$ is non-trivial, as $a$ is identified with $t_e$. Hence $C^0(\rho)$ is a subgroup of index 2 in $C(\rho)$. 

For $z\in F_n$ let $\gamma_z$ be the endomorphism of $F_n$ defined by:
$$\gamma_z(x)=\begin{cases}
              az ,&\text{if } x=a,\\
	      z^{-1}bz ,&\text{if } x=b,\\
	      c_i ,&\text{if } x=c_i.
              \end{cases}$$
For notational convenience we sometimes write $(\gamma_\bullet;z)$ instead of $\gamma_z$. We will always choose $z$ so that $\gamma_z$ is an automorphism.
\begin{theorem}\label{pres:c}
The centraliser in $\aut(F_n)$ of the Nielsen automorphism $\rho$ is generated by the following elements:
\begin{tabbing}
$P_{i,j}$\hspace{1cm}\=for $1\le i,j\le n-2$ and $i\neq j$,\\
$I_i$\>for $1\le i\le n$,\\
$(c_i^\epsilon;z)$\>for $1\le i\le n-2$, $\epsilon=\pm1$, and $c_i\neq z\in\{aba^{-1},b,c_1,\ldots,c_{n-2}\}$,\\
$\gamma_z$\>for $z\in\{aba^{-1},c_1,\ldots,c_{n-2}\}$,\\
$(a^{-1};z)$\>for $z\in\{b,c_1,\ldots,c_{n-2}\}$,\\
$\rho$,\\
$\theta$.
\end{tabbing}
For all possible choices $z,z_i\in\{aba^{-1},b,c_1,\ldots,c_{n-2}\}$ and $u,u_i\in\{c_1^{\pm1},\ldots,c_{n-2}^{\pm1},a^{-1},\gamma_\bullet\}$, whenever the symbols involved give well-defined generators or inverses from the list above, the following list is a collection of defining relators:
\begin{tabbing}
R1\qquad\=Relators in $\aut(F_{n-2})$ for $\{(c_i^\epsilon;c_j),P_{i,j},I_j\}$,\\
R2.1\>$(c_i^\epsilon;z_1)(c_j^\eta;z_2)=(c_j^\eta;z_2)(c_i^\epsilon;z_1)$ for $c_i^\epsilon\neq c_j^\eta$ and $z_1^{\pm1},z_2^{\pm1}\notin\{c_i,c_j\},$\\
R2.2\>$(c_i^\epsilon;z_1)(a^{-1};z_2)=(a^{-1};z_2)(c_i^\epsilon;z_1)$ for $z_j^{\pm1}\notin\{aba^{-1},c_i\}$,\\
R2.3\>$(c_i^\epsilon;z_1)\gamma_{z_2}=\gamma_{z_2}(c_i^\epsilon;z_1)$ for $z_j^{\pm1}\notin\{b,c_i\}$,\\
R2.4\>$(a^{-1};c_i)\gamma_{c_j}=\gamma_{c_j}(a^{-1};c_i)$,\\
R3.1\>$(a^{-1};c_j)P_{j,l}=P_{j,l}(a^{-1};c_l)$,\\
R3.2\>$(a^{-1};c_j)I_j=I_j(a^{-1};c_j^{-1})$,\\
R3.3\>$\gamma_{c_j}\circ P_{j,l}=P_{j,l}\circ\gamma_{c_l}$,\\
R3.4\>$\gamma_{c_j}\circ I_j=I_j\circ\gamma_{c_j}^{-1}$,\\
R3.5\>$P_{j,l}$, $I_j$ commute with $(a^{-1};b)$ and $\gamma_{aba^{-1}}$,\\
R3.6\>$(c_j^\epsilon;z)P_{j,l}=P_{j,l}(c_l^\epsilon;z)$ for $z=aba^{-1}$ or $z=b$,\\
R3.7\>$(c_j;z)I_j=I_j(c_j^{-1};z)$ for $z=aba^{-1}$ or $z=b$,\\
R4.1\>$(u;c_j^{-\eta})(c_j^\eta;z)(u;c_j^\eta)=(u;z)(c_j^\eta;z)$,\\
R4.2\>$(a^{-1};z^{-\epsilon})(u;aba^{-1})(a^{-1};z^\epsilon)=(u;z^\epsilon)(u;aba^{-1})(u;z^{-\epsilon})$,\\
R4.3\>$\gamma_z^{-\epsilon}(u;b)\gamma_z^\epsilon=(u;z^\epsilon)(u;b)(u;z^{-\epsilon})$,\\
R5.1\>$(c_j^{-\eta};aba^{-1})(a^{-1};c_j^\eta)=(a^{-1};c_j^\eta)(c_j^\eta;ab^{-1}a^{-1})\rho$,\\
R5.2\>$(c_j^{-\eta};b)\gamma_{c_j}^\eta\circ\rho=\gamma_{c_j}^\eta(c_j^\eta;b^{-1})$,\\
R6\>$\rho$ commutes with all generators,\\
R7\>$\theta^2=1$,\\
R8.1\>$\theta\circ P_{i,j}=P_{i,j}\circ\theta$,\\
R8.2\>$\theta\circ I_i=I_i\circ\theta$,\\
R8.3\>$\theta\circ(c_i^\epsilon;c_j)=(c_i^\epsilon;c_j)\circ\theta$,\\
R8.4\>$\theta\circ(c_i^\epsilon;aba^{-1})=(c_i^\epsilon;b^{-1})\circ\theta$,\\
R8.5\>$\theta\circ\gamma_{c_i}=(a^{-1};c_i)\circ\theta$,\\
R8.6\>$\theta\circ\gamma_{aba^{-1}}=(a^{-1};b^{-1})\circ\theta$.
\end{tabbing}
\end{theorem}

\begin{proof}
Recall the short exact sequence \eqref{sesn} on page \pageref{sesn}: $$1\to\langle\rho\rangle\to C^0(\rho)\to\aut(G_v,\C_v)\to1.$$ To get a presentation for $C^0(\rho)$, we now use Proposition \ref{l:sespres}. We first have to lift the generators of $\aut(G_v,\mathcal{C}_v)$ in Proposition \ref{p:JW} to $C^0(\rho)$. Since the surjection in the short exact sequence is given by restriction to the vertex group $G_v$, lifting means extending automorphisms from $G_v$ to all of $F_n$. The generators $P_{i,j}$, $I_i$, and $(c_i^\epsilon;c_j)$ are lifted to the elements of $C^0(\rho)$ called by the same name.
The elements $(c_i^\epsilon;y_1)$ and $(c_i^\epsilon;y_2)$ are lifted to $(c_i^\epsilon;aba^{-1})$ and $(c_i^\epsilon;b)$ respectively. The generator $(y_1^\pm;z)$ can be extended to $(a^{-1};z)$ and $(y_2^\pm;z)$ to $\gamma_z$. 

We now have a generating set for $C^0(\rho)$ consisting of these lifted generators together with $\rho$. The lifted relators are R1 through R5, which have (roughly) the same numbers as the corresponding relators in Proposition \ref{p:JW}. A direct inspection shows that $\rho$ only appears in R5. Since $\rho$ is central, the conjugation relators are simply the commutation rules R6. As the left hand term in this exact sequence is simply the infinite cyclic group generated by $\rho$, there are no kernel relators.

To get a presentation for $C(\rho)$, we apply Proposition \ref{l:sespres} to the short exact sequence $$1\to C^0(\rho)\to C(\rho)\overset{\bar\beta}{\to}\mathbb{Z}/2\mathbb{Z}\to1,$$ where $\mathbb{Z}/2\mathbb{Z}$ is identified with $\aut(\G,v)$, so that $\bar\beta(\theta)=-1$. A generating set for $C(\rho)$ is then given by $\theta$ and our chosen generators of $C^0(\rho)$. The relators are again R1 through R6 along with the lifted relator $\theta^2=1$ coming from $\mathbb{Z}/2\mathbb{Z}$, which we label R7, and the conjugation relators given by R8.
\end{proof}

\subsection{The abelianisation}

For an element $g$ in an arbitrary group $G$ let $\llbracket g\rrbracket$ denote its class in the abelianisation $H_1(G)=G/[G,G]$. We now study the abelianisation of $C(\rho)$.

\begin{corollary}\label{pres:cab} Let $\rho\in\aut(F_n)$ be a Nielsen automorphism. Then:
$$H_1(C(\rho))\cong \begin{cases}
\mathbb{Z}^2\oplus\mathbb{Z}/2\mathbb{Z} ,&\text{if } n=2,\\
\mathbb{Z}\oplus(\mathbb{Z}/2\mathbb{Z})^3 ,&\text{if } n=3,\\
(\mathbb{Z}/2\mathbb{Z})^3 ,&\text{if } n=4,\\
(\mathbb{Z}/2\mathbb{Z})^2 ,&\text{if } n\ge5.
\end{cases}$$
When $n=2$, the class $\llbracket\rho\rrbracket$ is a generator of $\mathbb{Z}^2$, when $n=3$ it is twice a generator of $\mathbb{Z}$, and otherwise $\llbracket\rho\rrbracket=0$.
\end{corollary}

\begin{proof}
We abelianise the presentation in Theorem~\ref{pres:c}. We first restrict to the case $n=2$: the generators of $C(\rho)$ in this case are $\gamma_{aba^{-1}}$, $(a^{-1};b)$, $\rho$ and $\theta$. The only relators which occur and are non-trivial in the abelianisation are R7 and R8.6, which become $2\llbracket\theta\rrbracket=0$ and $\llbracket\gamma_{aba^{-1}}\rrbracket+\llbracket(a^{-1};b)\rrbracket=0$. This finishes the proof of the assertion for $n=2$.

Next we consider $n=3$. For simplicity we write $c:=c_1$. Here the generators of $C(\rho)$ are $I:=I_1$, $(c^\epsilon;aba^{-1})$, $(c^\epsilon;b)$, $\gamma_{aba^{-1}}$, $\gamma_c$, $(a^{-1};b)$, $(a^{-1};c)$, $\rho$ and $\theta$. From Theorem~\ref{pres:c} we obtain the following relators:
\begin{align*} \text{R1:} && 2\llbracket I\rrbracket&=0  &\text{R4.1:} &&  \llbracket(a^{-1};b)\rrbracket&=0 \\
\text{R3.2:} && 2\llbracket(a^{-1};c)\rrbracket&=0 &\text{R5.1:} && \llbracket\rho\rrbracket&=2\llbracket(c;aba^{-1})\rrbracket\\
\text{R3.4:} && 2\llbracket\gamma_c\rrbracket&=0 &\text{R5.2:} && -\llbracket\rho\rrbracket&=2\llbracket(c;b)\rrbracket\\
\text{R3.7:} && \llbracket(c;aba^{-1})\rrbracket&=\llbracket(c^{-1};aba^{-1})\rrbracket &\text{R7:}  && 2\llbracket\theta\rrbracket&=0\\
\text{R3.7:} &&  \llbracket(c;b)\rrbracket&=\llbracket(c^{-1};b)\rrbracket &\text{R8.4:} && \llbracket(c;b)\rrbracket&=-\llbracket(c;aba^{-1})\rrbracket \\
\text{R4.1} && \llbracket\gamma_{aba^{-1}}\rrbracket&=0 &\text{R8.5:} &&  \llbracket\gamma_c\rrbracket&=\llbracket(a^{-1};c)\rrbracket \end{align*}

All other relators in $H_1(C(\rho))$ either follow from the ones above or are trivial. It follows that $H_1(C(\rho)) \cong \mathbb{Z}\oplus(\mathbb{Z}/2\mathbb{Z})^3$ with the torsion part generated by $\llbracket I\rrbracket$, $\llbracket\gamma_c\rrbracket$ and $\llbracket\theta\rrbracket$ and the torsion-free part generated by  $\llbracket(c;b)\rrbracket$ with $\llbracket\rho\rrbracket=-2\llbracket(c;b)\rrbracket$.

For $n \geq 4$, by checking the relators R1--R8 one finds that there is a homomorphism
$$C(\rho) \to H_1(\aut(F_{n-2})) $$
given by sending the elements $(c_i^\epsilon;aba^{-1})$, $(c_i^\epsilon;b)$, $\g_z$, $(a^{-1};z)$, $\rho$ and $\theta$ to $0$ and letting the remaining generators of $C(\rho)$ act on $F_{n-2}=\langle c_1,\ldots,c_{n-2} \rangle$. We also have the homomorphism $$C(\rho) \to C(\rho)/C^0(\rho)\cong \mathbb{Z}/2\mathbb{Z}$$ that takes every generator except $\theta$ to $0$. Combining these gives a surjective homomorphism $$f:C(\rho) \to H_1(\aut(F_{n-2})) \oplus C(\rho)/C^0(\rho).$$ As $\text{Im} f$ is abelian, this descends to a surjective map: $$f_*:H_1(C(\rho)) \to H_1(\aut(F_{n-2})) \oplus C(\rho)/C^0(\rho).$$

The relation R4.1 implies that $\llbracket (u ; z) \rrbracket = 0$ if there is a symbol $c_j$ different from both $u,z$ and their inverses. Hence any generator $(u ;z)$ not of the form $(c_i^\epsilon,c_j)$ is trivial in $H_1(C(\rho))$. Furthermore, as $\llbracket (c_i;aba^{-1}) \rrbracket=0$ we have $\llbracket \rho \rrbracket =0$ by R5.1. It follows that every $g \in H_1(C(\rho))$ is represented by a word of the form $w\cdot\theta^\delta$, where $w$ is a product of elements of the form $(c_i^\epsilon,c_j)$ and $\delta \in \{0,1\}$. However, if $f_*(g)=0$, this implies that $w$ is a product of commutators in $\aut(F_{n-1})$ and $\delta = 0$, so that $g$ is trivial in $H_1(C(\rho))$. Hence $f_*$ is also injective and is an isomorphism. We finish the proof with the observation that:
\begin{equation*} H_1(\aut(F_{n-2})) = \begin{cases} \mathbb{Z} /2 \mathbb{Z} \oplus \mathbb{Z} / 2\mathbb{Z}, & \text{if $n=4$,} \\ \mathbb{Z}/2\mathbb{Z}, &\text{if $n \geq 5$ .}  \end{cases} \end{equation*}
This may be found by abelianising one's favourite presentation of $\aut(F_n)$ (from \cite{M2}, \cite{Ne}, or \cite{Ni}, say).
\end{proof}

\begin{remark}
Given a basis $a_1,\ldots,a_n$ of $F_n$ and $w\in\langle a_2,\ldots,a_n\rangle$, we can define $\rho\in\aut(F_n)$ by $a_1\to a_1w$ and fixing $a_2,\ldots,a_n$. Abelianisations of centrailsers of these more general right translations are computed in \cite{R}.
\end{remark}

\subsection{Nielsen automorphisms in $\out(F_n)$}

Although the above work has focused on $\aut(F_n)$, the centraliser of the outer automorphism class $\hat\rho\in\out(F_n)$ of a Nielsen automorphism is closely related.

\begin{lemma}\label{l:cautout}
Let $D$ and $\mathcal G$ be as in Section \ref{s:nielsen_dt}, so that $D_{*v}=\rho$ is a Nielsen automorphism and $\widehat{D}=\hat{\rho}$ is its image in $\out(F_n)$. The natural homomorphism
$$C(\rho)\to C(\hat\rho)$$
is surjective with kernel equal to $\ad(F_n)\cap C(\rho)=\ad(G_v)$.
\end{lemma}

\begin{proof}
The two short exact sequences of Theorems \ref{t:pses} and \ref{t:ses} fit into a commutative diagram:
$$\xymatrix{DA(\mathcal{G})\ar@{^(->}[r]\ar[d]^\cong&C^0(\rho)\ar@{->>}[r]\ar[d]&\aut(G_v,\C_v)\ar@{->>}[d]\\
DO(\mathcal{G})\ar@{^(->}[r]&C^0(\hat{\rho})\ar@{->>}[r]&\out(G_v,\mathcal{C}_v)}$$
By diagram chase we see that the natural map $C^0(\rho)\to C^0(\hat \rho)$ is onto. Together with the homomorphisms $\bar\alpha$ from Theorem \ref{t:ses} and $\bar\beta$ from Theorem \ref{t:pses} we get the commutative diagram
$$\xymatrix{C^0(\rho)\ar@{^(->}[r]\ar@{->>}[d]&C(\rho)\ar[r]^-{\bar\beta}\ar[d]&\aut(\G,v)\ar@{^(->}[d]\\
C^0(\hat \rho)\ar@{^(->}[r]&C(\hat \rho)\ar[r]^-{\bar\alpha}&\aut(\G)}$$
We have already seen that $\bar\beta(\theta)$ is non-trivial in $\aut(\G,v)=\aut(\G)\cong\mathbb Z/2\mathbb Z$, so both $\bar\alpha$ and $\bar\beta$ are surjective. Hence the vertical map in the middle is surjective, too. The kernel of the map $C(\rho) \to C(\hat\rho)$ consists of the inner automorphisms in $C(\rho)$. Proposition 7.2 of \cite{CL} tells us that the subgroup of $F_n$ fixed by $\rho$ is $G_v = \langle aba^{-1} , b, c_2, \ldots, c_n \rangle$, and an inner automorphism $ad_g$ commutes with $\rho$ if an only if $g$ is fixed by $\rho$. It follows that $\ad(F_n)\cap C(\rho)=\ad(G_v)$.
\end{proof}

\begin{corollary} The centraliser of a Nielsen automoprhism $\hat\rho$ in $\out(F_n)$ has a presentation consisting of the 
generators and relations given in Theorem~\ref{pres:c}, with the following additional relators:
\begin{tabbing} 
R9.1\qquad\=$\gamma_{aba^{-1}}\rho^{-1}\cdot\prod_{i=1}^{n-2}(c_i;aba^{-1})(c_i^{-1};aba^{-1})=1$,\\
R9.2\> $(a^{-1};b)\rho\cdot\prod_{i=1}^{n-2}(c_i;b)(c_i^{-1};b)=1,$\\
R9.3\> $\gamma_{c_i}(a^{-1};c_i)\cdot \prod_{j \neq i} (c_j;c_i)(c_{j}^{-1};c_i)=1$ for $1 \leq i \leq n-2$.
\end{tabbing}
\end{corollary}

\begin{proof}
By Lemma~\ref{l:cautout}, the kernel of the surjective map $C(\rho) \to C(\hat\rho)$ is $\ad(F_n)\cap C(\rho)=\ad(G_v)$. This is the free group generated by $\ad_{aba^{-1}}$, $ad_b$, and $ad_{c_i}$ for $1 \leq i \leq n-2$. One needs only to add relators to the presentation of $C(\rho)$ corresponding to these elements. These are given by R9.1, R9.2, and R9.3 respectively.
\end{proof}

\begin{corollary} Let $\hat\rho\in\out(F_n)$ be the outer automorphism class of a Nielsen automorphism. Then:
$$H_1(C(\hat\rho))\cong \begin{cases}
\mathbb{Z}\oplus\mathbb{Z}/2\mathbb{Z} ,&\text{if } n=2,\\
\mathbb{Z}\oplus(\mathbb{Z}/2\mathbb{Z})^3 ,&\text{if } n=3,\\
(\mathbb{Z}/2\mathbb{Z})^3 ,&\text{if } n=4,\\
(\mathbb{Z}/2\mathbb{Z})^2 ,&\text{if } n\ge5.
\end{cases}$$
When $n=2$, the class $\llbracket\rho\rrbracket$ is a generator of $\mathbb{Z}$, when $n=3$ it is twice a generator of $\mathbb{Z}$, and otherwise $\llbracket\rho\rrbracket=0$.
\end{corollary}
\begin{proof}
When $n=2$, in the proof of Corollary~\ref{pres:cab} we found that $H_1(C(\rho))$ has a free abelian subgroup generated by $\llbracket \rho \rrbracket$ and $\llbracket \g_{aba^{-1}} \rrbracket$, and that $\llbracket \g_{aba^{-1}} \rrbracket = - \llbracket (a^{-1};b) \rrbracket$. The new relators R9.1 and R9.2 reduce to $\llbracket \g_{aba^{-1}} \rrbracket = \llbracket \rho \rrbracket$ and $\llbracket (a^{-1};b) \rrbracket=- \llbracket \rho \rrbracket$ in the abelianisation, so that when we pass to $\out(F_n)$ we only have a rank one free abelian factor. 

When $n \geq 3$, one can use the work in the proof of Corollary~\ref{pres:cab} to check that $\llbracket\ad_{aba^{-1}}\rrbracket=\llbracket\ad_b\rrbracket=\llbracket\ad_{c_i}\rrbracket=0$. It follows that the natural map $C(\rho)\to C(\hat\rho)$ induces an isomorphism on abelianisations.
\end{proof}

\subsection{Connection to CAT(0) actions}

A {\em CAT(0) space} is a geodesic metric space $(X,d)$ whose geodesic triangles are not thicker than euclidean comparison triangles with the same sidelengths (cf. \cite{BH} for a more precise definition).  A metric space is \emph{proper} if all of its closed balls are compact. The {\em translation length} of an isometry $\gamma:X\rightarrow X$ is defined by
$$|\gamma|=\inf_{x\in X}d(x,\gamma(x)).$$
The work above relates to the following theorem, which appears in the proof of Theorem 2.6 in \cite{B:mcg}:

\begin{theorem}[Bridson, Karlsson, Margulis]
Let $G$ be any group and $g\in G$. Assume that $\llbracket g\rrbracket$ has finite order in $H_1(C(g))$. Then $|g|=0$ whenever $G$ acts by isometries on a proper CAT(0) space.
\end{theorem}

In Corollary \ref{pres:cab} we have seen that $\llbracket\rho\rrbracket$ has infinite order in $H_1(C(\rho))$ if $n\le3$. In fact, in Section 6 of \cite{B:rd} there is a construction for isometric CAT(0) actions of $\aut(F_3)$ such that Nielsen automorphisms act by positive translation length. However, for $n\ge4$ we have seen that $\llbracket\rho\rrbracket=0$. Hence:
\begin{corollary}\label{c:transl}
If $n\ge4$, Nielsen automorphisms always act by zero translation length whenever $\aut(F_n)$ acts isometrically on a proper CAT(0) space.
\end{corollary}

This sharpens a result of Bridson \cite{B:rd}, who proved Corollary \ref{c:transl} when $n \geq 6$ by showing that $\llbracket\rho\rrbracket=0$ when $n\ge6$ using the lantern relation in mapping class groups, rather than a direct computation of $H_1(C(\rho))$.

The thesis of the first author contains a short geometric proof of Corollary \ref{c:transl}. Here is a sketch of the proof:
\begin{proof}[Alternative proof of Corollary~\ref{c:transl}]
For simplicity, we restrict to the $\aut(F_4)$ case. Two commuting isometries $\alpha$ and $\beta$ in a CAT(0) space satisfy the \emph{parallelogram formula}:
$$|\alpha\beta|^2+|\alpha\beta^{-1}|^2=2(|\alpha|^2+|\beta|^2).$$
(This is proved in \cite{R}.) If we define $\alpha$ and $\beta$ like so:
\begin{align*}
\alpha\co a&\mapsto a, & \beta\co &a \mapsto ab,\\
b&\mapsto b, & &b \mapsto b,\\
c&\mapsto cb, & &c \mapsto c,\\
d&\mapsto db^{-1}, &&d \mapsto d, 
\end{align*}
then $\alpha$ is then conjugate to both $\alpha\beta$ and $\alpha\beta^{-1}$. Hence $\alpha$, $\alpha\beta$, and  $\alpha\beta^{-1}$ have the same translation lengths. The parallelogram formula then implies that $|\beta|=0$.
\end{proof}
\bibliographystyle{plain}

\vspace{2cm}

Moritz Rodenhausen\\Mathematisches Institut\\Endenicher Allee 60\\53115 Bonn\\Germany\\moritz.rodenhausen@gmx.de\\

Richard D. Wade\\Math Department\\The University of Utah\\155 S 1400 E\\Salt Lake City\\Utah 84112\\USA\\wade@math.utah.edu\\

\end{document}